\documentclass[11pt]{article}
\usepackage{caption}
\usepackage{url}
\usepackage{verbatim}
\usepackage{enumitem}
\usepackage{amsmath}
\usepackage[usenames,dvipsnames]{color}
\usepackage{amsthm}
\usepackage{amsfonts}
\usepackage{graphicx}
\usepackage{nicefrac}
\usepackage{url}
\usepackage{etoolbox}
\usepackage{algorithm}
\usepackage{algpseudocode}
\usepackage{graphicx, amssymb, subcaption,graphics}
\def\captionfont{\setb@se{11pt}\protect\footnotesize}
\def\captionfont{\protect\footnotesize}
\usepackage{tikz}
\usetikzlibrary{arrows}
\usepackage{verbatim}

\newcommand{\vertiii}[1]{{\left\vert\kern-0.25ex\left\vert\kern-0.25ex\left\vert #1 \right\vert\kern-0.25ex\right\vert\kern-0.25ex\right\vert}}

\newcommand{\msfv}{\mathsf{v}}

\newcommand{\mD}{\mathfrak D}
\newcommand{\md}{\mathfrak d}

\newcommand{\hf}{\frac{1}{2}}

\def\0{\mbox{\boldmath $0$}}

\newcommand{\nrm}[1]{\left\| #1 \right\|}
\newcommand{\ciptwo}[2]{\left( #1 , #2 \right)}

\newcommand{\eipns}[2]{\left[ #1 , #2 \right]_{\rm ns}}
\newcommand{\eipew}[2]{\left[ #1 , #2 \right]_{\rm ew}}
\newcommand{\viptwo}[2]{\left\langle #1, #2 \right\rangle}
\newcommand{\nabh}{\nabla_{\! h}}

	\newcommand\be {\begin{equation}}
	\newcommand\ee {\end{equation}}

	\newcommand\bx {{\bf x}}
	
	\newcommand\dt {{s}}
 \def\P{\mbox{$\mathsf{P}_h$}}
 
\newtheorem{thm}{Theorem}[section]
\newtheorem{prop}[thm]{Proposition}

\newtheorem{rmk}[thm]{Remark}
\newtheorem{lem}[thm]{Lemma}
\newtheorem{rem}[thm]{Remark}

\allowdisplaybreaks	

\begin{document}

\title{A Second-Order Energy Stable Backward Differentiation Formula Method for the Epitaxial Thin Film Equation with Slope Selection}
\author{
Wenqiang Feng\thanks{Department of Mathematics, The University of Tennessee, Knoxville, TN 37996 (Corresponding Author: wfeng1@utk.edu)}
\and
Cheng Wang\thanks{Department of Mathematics, The University of Massachusetts, North Dartmouth, MA  02747 (cwang1@umassd.edu)}	
\and
Steven M. Wise\thanks{Department of Mathematics, The University of Tennessee, Knoxville, TN 37996 (swise1@utk.edu)}
\and
Zhengru Zhang\thanks{School of Mathematical Sciences, Beijing Normal University, Beijing 100875, P.R. China (zrzhang@bnu.edu.cn)} 
}

\maketitle
\numberwithin{equation}{section}

	\begin{abstract}
In this paper, we study a novel second-order energy stable Backward Differentiation Formula (BDF) finite difference scheme for the epitaxial thin film equation with slope selection (SS). One major challenge for the higher oder in time temporal discretization is how to ensure an unconditional energy stability and an efficient numerical implementation. We propose a general framework for designing the higher order in time numerical scheme with unconditional energy stability by using the BDF method with constant coefficient stabilized terms. Based on the unconditional energy stability property, we derive an $L^\infty_h (0,T; H_{h}^2)$ stability for the numerical solution and provide an optimal the convergence analysis. To deal with the 4-Laplacian solver in an $L^{2}$ gradient flow at each time step, we apply an efficient preconditioned steepest descent algorithm and preconditioned nonlinear conjugate gradient algorithm to solve the corresponding nonlinear system. Various numerical simulations are present to demonstrate the stability and efficiency of the proposed schemes and solvers.
	\end{abstract}

	\section{Introduction}
	
In this paper we will devise and analyze numerical methods for the epitaxial thin film model with slope selection, or, for short, just the slope selection (SS) equation. This equation is the gradient flow with respect to the energy
	\begin{equation}
F [\phi] := \int_\Omega \left( \frac14 \left(| \nabla \phi |^2 -1\right)^2 +\frac{\varepsilon^2}{2} ( \Delta \phi )^2  
	 \right) \,\mathrm{d}\bx \ , 
		\label{energy-SS-1} 
	\end{equation}
where $\Omega = (0, L_x)\times (0, L_y)$, $\phi:\Omega\rightarrow \mathbb{R}$ is a scaled height function of thin film and $\varepsilon$ is a constant which represents the width of the rounded corner. As is common, and natural, we will assume that $\phi$ is $\Omega$-periodic. The corresponding chemical potential  is defined to be the variational derivative 
	of the energy (\ref{energy-SS-1}), \emph{i.e.}, 
	\begin{equation}
	\mu: = \delta_\phi F =  - \nabla \cdot ( | \nabla \phi |^2 \nabla \phi ) + \Delta \phi + \varepsilon^2 \Delta^2 \phi .
		\label{chem-pot-SS}
	\end{equation}
In turn, the SS equation becomes the $L^2$ gradient flow associated with the energy (\ref{energy-SS-1}): 
	\begin{equation}
\partial_t \phi = - \mu = \nabla \cdot ( | \nabla \phi |^2 \nabla \phi ) - \Delta \phi - \varepsilon^2 \Delta^2 \phi . 
	\label{equation-SS}
	\end{equation}
The SS equation was proposed by P. Aviles and Y. Giga to study the dynamics of smectic liquid crystals in \cite{aviles1987mathematical}. Since then, it has attracted considerable attention in several related fields, for instance as a model for the deformation of thin film blisters \cite{ortiz1994morphology}, the delamination of compressed thin films \cite{gioia1997delamination}, the line energies for gradient vector fields in the plane  \cite{ambrosio1999line}, and the domain wall energy in a problem related to micro-magnetism \cite{riviere2001domain}.

In the thin film setting, the energy of the SS equation can be considered as two distinct parts. The first part is 
	\begin{equation}
	\label{eqn:Eenergy-ES}
F_{\rm ES}[\phi] :=  \int_\Omega \frac14 \left(| \nabla \phi |^2 -1\right)^2 \mathrm{d}\bx,
	\end{equation}
which describes, in some limited sense, the Ehrlich-Schwoebel effect--the phenomenon where atoms tend to move from a lower terrace to an upper terrace in the growth of atomic steps, promoting surface instability. Mathematically, the term $F_{\rm ES}$ gives the preference for epitaxial films with slope satisfying $|\nabla\phi|=1$, since this represents the minima of $F_{\rm ES}$.  The second part,  
 \begin{eqnarray}\label{eqn:Eenergy-SD}
 F_{\rm SD}[\phi] =  \int_\Omega \frac{\varepsilon^2}{2}  (\Delta \phi)^2  \mathrm{d}\bx, 
 \end{eqnarray}
 represents the surface diffusion effect which will give the rounded corners in the film. A smaller value of $\varepsilon$ corresponds to a sharper rounded corner. There are some other interesting physical predictions coming from the SS model; for instance, the surface roughness grows approximately at the rate $t^{1/3}$, the energy decays at approximately the rate $t^{-1/3}$ (see \cite{kohn03}), and the saturation time scale is expected to be the order of $\varepsilon^{-2}$ (see \cite{wise09a}). An energy stable scheme with higher order temporal accuracy has always been highly desirable because these processes are realized only in the sense of very large times.
 
 There have been several works focused on second-order-in-time schemes for the SS equation in recent years. In \cite{xu06}, the authors proposed a hybrid scheme, one which combined a second-order backward differentiation for the time derivative term and a second-order extrapolation for the explicit treatment of the nonlinear term. A linear stabilization parameter $A$ has to be sufficiently large to guarantee the energy dissipation law for this scheme, and a theoretical justification of the lower bound for $A$ has not been available. As an alternate approach, a second-order-in-time operator splitting scheme was proposed for the SS equation in \cite{cheng2015fast}, in which the nonlinear part is solved using the fourth-order central difference approximation combined with the third-order explicit Runge-Kutta method. The corresponding convergence analysis was provided in \cite{li2015convergence}. Similar operator splitting ideas can also be found in a recent work  \cite{lee2017second}. Some other second-order-in-time numerical approaches were reported in recent years, such as a linearized finite difference scheme in \cite{qiao2012stability}, an adaptive time-stepping strategy with Crank-€"Nicolson (CN) formulas in \cite{qiao2011adaptive}, the BDF and the CN formulas with invariant energy quadratization strategy proposed in \cite{yang2016numerical}.
 
Meanwhile, it is observed that, the long time energy stability could not be theoretically justified for these numerical works, due to the explicit treatment for the nonlinear terms. In the existing literature, the only second-order-in-time numerical algorithm for the SS equation (\ref{equation-SS}) with a long time energy stability could be found in \cite{shen2012second}, in which modified Crank-Nicolson approximations are used for the nonlinear 4-Laplacian term and the surface diffusion term, while an explicit extrapolation formula is applied to the concave diffusion term, respectively; in turn, an unconditional energy stability is derived as a result of a careful energy estimate. On the other hand, extensive numerical experiments have indicated a fairly poor performance for the modified Crank-Nicolson scheme, due to the highly nonlinear nature of the 4-Laplacian term, as well as the complicated form involved with the CN approximation.

Any numerical scheme that treats the nonlinear terms implicitly -- for the purposes of accuracy or stability, or both -- requires one to solve a regularized 4-Laplacian-type equation, where the highest-order term is a linear biharmonic operator. Consequently, an efficient solver for a regularized p-Laplacian equation has always been highly desirable. In a recent work~\cite{feng2016preconditioned}, a preconditioned steepest descent (PSD) algorithm was proposed for such problems. At each iteration stage, only a purely linear elliptic equation needs to be solved to obtain a search direction, and the numerical efficiency for such an elliptic equation could be greatly improved with the use of FFT-based solvers. In turn, an optimization in the given search direction becomes one-dimensional, with its well-posedness assured by convexity arguments. Moreover, a geometric convergence of such an iteration could be theoretically derived, so that a great improvement of the numerical efficiency is justified, in comparison with an application of the Polak-Ribi\'ere variant of NCG (nonlinear conjugate gradient) method \cite{polak69}, reported in \cite{shen2012second, wang10a}. 

The PSD algorithm has been very efficiently applied to the first order energy stable scheme for the SS equation, as reported in \cite{feng2016preconditioned}. However, its application to the CN version of the second order energy stable scheme, as proposed in~\cite{shen2012second}, has faced serious difficulties. These difficulties come from a subtle fact that, the modified CN approximation to the 4-Laplacian term does not correspond to a convex energy functional, because of the vector gradient form (other than a scalar form) in the 4-Laplacian expansion. Consequently, a natural question arises: could the PSD solver be efficiently combined with a second order energy stable scheme for the SS equation (\ref{equation-SS})? In this article, we propose a second order BDF scheme for the SS equation (\ref{equation-SS}), so that the unique solvability, energy stability could be theoretically derived, and the PSD solver could be efficiently applied. In more details, an alternate second order energy stable scheme is proposed, based on the 2nd order BDF temporal approximation framework, instead of that based on the CN one. The 2nd order BDF scheme treats and approximates every term at the time step $t^{k+1}$ (instead of the time instant $t^{k+1/2}$): a 2nd order BDF 3-point stencil is applied in the temporal derivative approximation, the nonlinear term and the surface diffusion terms are updated implicitly for their strong convexities, and a second order accurate, explicit extrapolation formula is applied in the approximation of the concave diffusion term. Such a structure makes the numerical scheme uniquely solvable. In addition, to ensure the energy stability of the numerical scheme, we need to add a second order Douglas-Dupont regularization, in the form of $-A \tau \Delta ( \phi^{k+1} - \phi^k )$. We prove that, under a mild requirement $A \ge \frac{1}{16}$, rigorous energy stability is guaranteed.

In fact, the 2nd order accurate, energy stable BDF scheme for the Cahn-Hilliard model was analyzed in a recent article~\cite{yan16a} with similar ideas. In particular, the nonlinear solver required for the BDF scheme is reported to require 20 to 25 percent less computational effort than that for the Crank-Nicolson version, due to the simpler form and stronger convexity properties of the nonlinear term. For the SS equation (\ref{equation-SS}), a much greater improvement (in terms of numerical efficiency) is expected for the BDF approach, due to the more complicated form of the 4-Laplacian term. Based on the unconditional energy stability, we derive an $L_h^\infty (0,T; H_{\rm per}^2)$ stability for the numerical solution. In turn, with the help of Sobolev embedding from $H_{\rm per}^2$ into $W^{1,6}$, we prove the convergence of the proposed scheme.

The remainder of the paper is organized as follows. In Section~\ref{sec:fdm}, we present the discrete spatial difference operators, function space, inner products and norms, define the proposed second-order-in-time fully discrete finite difference scheme and prove that the scheme is unconditionally stable and uniquely solvable, provide that the stabilized parameter $A\geq \nicefrac{1}{16}$. In Section~\ref{sec:convergence}, we provide a rigorous convergence analysis and error estimate for the proposed scheme. The preconditioned steepest descent solver and and preconditioned nonlinear conjugate gradient solver are outlined in Section~\ref{sec:psd} and Section~\ref{sec:pncg}, respectively. Finally, numerical experiments are presented Section~\ref{sec:num}, and some concluding remarks are given in Section~\ref{sec:conclusion}.
	\section{ The Fully Discrete Scheme with Finite Difference Spatial Discretization in 2D}
	\label{sec:fdm}
	\subsection{Notation}
	\label{sec:notation}	
In this subsection we define the discrete spatial difference operators, function space, inner products and norms, following the notation used in \cite{feng2016fch,feng2016preconditioned, shen2012second, wang10a, wang11, wise09a}.  Let $\Omega = (0,L_x)\times(0,L_y)$, where, for simplicity, we assume $L_x =L_y =: L > 0$. We write $L = m\cdot h$, where $m$ is a positive integer. The parameter $h = \frac{L}{m}$ is called the mesh or grid spacing. We define the following two uniform, infinite grids with grid spacing $h>0$:
	\[
E := \{ x_{i+\hf} \ |\ i\in {\mathbb{Z}}\}, \quad C := \{ x_i \ |\ i\in {\mathbb{Z}}\},
	\]
where $x_i = x(i) := (i-\hf)\cdot h$. Consider the following 2D discrete periodic function spaces: 
	\begin{eqnarray*}
{\mathcal V}_{\rm per} &:=& \left\{\nu: E\times E\rightarrow {\mathbb{R}}\ \middle| \ \nu_{i+\frac12,j+\frac12}= \nu_{i+\frac12+\alpha m,j+\frac12+\beta m}, \ \forall \, i,j,\alpha,\beta\in \mathbb{Z}  \right\},
	\\
{\mathcal C}_{\rm per} &:=& \left\{\nu: C\times C
\rightarrow {\mathbb{R}}\ \middle| \ \nu_{i,j} = \nu_{i+\alpha m,j+\beta m}, \ \forall \, i,j,\alpha,\beta\in \mathbb{Z} \right\},
	\\
{\mathcal E}^{\rm ew}_{\rm per} &:=& \left\{\nu: E\times C\rightarrow {\mathbb{R}}\ \middle| \ \nu_{i+\frac12,j}= \nu_{i+\frac12+\alpha m,j+\beta m}, \ \forall \,  i,j,\alpha,\beta\in \mathbb{Z}  \right\},
	\\
{\mathcal E}^{\rm ns}_{\rm per} &:=& \left\{\nu: C \times E\rightarrow {\mathbb{ R}}\ \middle| \ \nu_{i,j+\frac12}= \nu_{i+\alpha m,j+\frac12+\beta m}, \ \forall \, i,j,\alpha,\beta\in \mathbb{Z}  \right\}.
	\end{eqnarray*}		
The functions of ${\mathcal V}_{\rm per}$ are called {\emph{vertex centered functions}}; those of ${\mathcal C}_{\rm per}$ are called {\emph{cell centered functions}}. The functions of ${\mathcal E}^{\rm ew}_{\rm per}$ are called {\emph{east-west edge-centered functions}}, and the functions of ${\mathcal E}^{\rm ns}_{\rm per}$ are called {\emph{north-south edge-centered functions}}.  We also define the mean zero space 
	\[
\mathring{\mathcal C}_{\rm per}:=\left\{\nu\in {\mathcal C}_{\rm per} \ \middle| \  \frac{h^2}{| \Omega|} \sum_{i,j=1}^m \nu_{i,j} =: \overline{\nu}  = 0\right\} .
	\]

We now introduce the important difference and average operators on the spaces:  
	\begin{eqnarray*}
&& A_x \nu_{i+\hf,\Box} := \frac{1}{2}\left(\nu_{i+1,\Box} + \nu_{i,\Box} \right), \quad D_x \nu_{i+\hf,\Box} := \frac{1}{h}\left(\nu_{i+1,\Box} - \nu_{i,\Box} \right),\\
&& A_y \nu_{\Box,i+\hf} := \frac{1}{2}\left(\nu_{\Box,i+1} + \nu_{\Box,i} \right), \quad D_y \nu_{\Box,i+\hf} := \frac{1}{h}\left(\nu_{\Box,i+1} - \nu_{\Box,i} \right) , 
	\end{eqnarray*}
with $A_x,\, D_x: {\mathcal C}_{\rm per}\rightarrow{\mathcal E}_{\rm per}^{\rm ew}$ if $\Box$ is an integer, and $A_x,\, D_x: {\mathcal E}^{\rm ns}_{\rm per}\rightarrow{\mathcal V}_{\rm per}$ if $\Box$ is a half-integer, with $A_y,\, D_y: {\mathcal C}_{\rm per}\rightarrow{\mathcal E}_{\rm per}^{\rm ns}$ if $\Box$ is an integer, and $A_y,\, D_y: {\mathcal E}^{\rm ew}_{\rm per}\rightarrow{\mathcal V}_{\rm per}$ if $\Box$ is a half-integer. Likewise,
	\begin{eqnarray*}
&&a_x \nu_{i,\Box} := \frac{1}{2}\left(\nu_{i+\hf,\Box} + \nu_{i-\hf,\Box} \right),	 \quad d_x \nu_{i,\Box} := \frac{1}{h}\left(\nu_{i+\hf,\Box} - \nu_{i-\hf,\Box} \right),
	\\
&&a_y \nu_{\Box,j} := \frac{1}{2}\left(\nu_{\Box,j+\hf} + \nu_{\Box,j-\hf} \right),	 \quad d_y \nu_{\Box,j} := \frac{1}{h}\left(\nu_{\Box,j+\hf} - \nu_{\Box,j-\hf} \right),
	\end{eqnarray*}
with $a_x,\, d_x : {\mathcal E}_{\rm per}^{\rm ew}\rightarrow{\mathcal C}_{\rm per}$ if $\Box$ is an integer, and $a_x,\ d_x: {\mathcal V}_{\rm per}\rightarrow{\mathcal E}^{\rm ns}_{\rm per}$ if $\Box$ is a half-integer; and with $a_y,\, d_y : {\mathcal E}_{\rm per}^{\rm ns}\rightarrow{\mathcal C}_{\rm per}$ if $\Box$ is an integer, and $a_y,\ d_y: {\mathcal V}_{\rm per}\rightarrow{\mathcal E}^{\rm ew}_{\rm per}$ if $\Box$ is a half-integer.
    
Also define the 2D center-to-vertex derivatives $\mD_x,\, \mD_y : {\mathcal C}_{\rm per}\rightarrow{\mathcal V}_{\rm per}$ component-wise as
	\begin{eqnarray*}
\mD_x \nu_{i+\hf,j+\hf} &:=& A_y(D_x\nu)_{i+\hf,j+\hf} = D_x(A_y\nu)_{i+\hf,j+\hf}
	\nonumber
	\\
&=& \frac{1}{2h}\left(\nu_{i+1,j+1}-\nu_{i,j+1}+\nu_{i+1,j}-\nu_{i,j} \right) ,
	\\
\mD_y\nu_{i+\hf,j+\hf} &:=& A_x(D_y\nu)_{i+\hf,j+\hf} = D_y(A_x\nu)_{i+\hf,j+\hf} 
	\nonumber
	\\
&=& \frac{1}{2h}\left(\nu_{i+1,j+1}-\nu_{i+1,j}+\nu_{i,j+1}-\nu_{i,j} \right) . 
	\end{eqnarray*}		

The utility of these definitions is that the differences $\mD_x$ and $\mD_y$ are collocated on the grid, unlike $D_x$, $D_y$. We denote the 2D vertex-to-center derivatives $\md_x,\, \md_y  : {\mathcal V}_{\rm per}\rightarrow{\mathcal C}_{\rm per}$  component-wise as
	\begin{eqnarray*}
\md_x \nu_{i,j} &:=& a_y(d_x \nu)_{i,j} = d_x(a_y \nu)_{i,j} 
	\nonumber
	\\
&=& \frac{1}{2h}\left(\nu_{i+\hf,j+\hf}- \nu_{i-\hf,j+\hf}+\nu_{i+\hf,j-\hf}-\nu_{i-\hf,j-\hf}\right) ,
	\\
\md_y \nu_{i,j} &:=& a_x(d_y \nu)_{i,j} = d_y(a_x \nu)_{i,j}
	\nonumber
	\\
&=& \frac{1}{2h}\left(\nu_{i+\hf,j+\hf}- \nu_{i+\hf,j-\hf}+\nu_{i-\hf,j+\hf}-\nu_{i-\hf,j-\hf}\right).
	\end{eqnarray*}

In turn, the discrete gradient operator, $\nabla^\msfv_h$: ${\mathcal C}_{\rm per}\rightarrow {\mathcal V}_{\rm per}\times {\mathcal V}_{\rm per}$, is defined as 
	\[
\nabla^\msfv_h \nu_{i+\hf,j+\hf} := (\mD_x \nu_{i+\hf,j+\hf}, \mD_y \nu_{i+\hf,j+\hf}). 
	\]	
The standard 2D discrete Laplacian, $\Delta_h : {\mathcal C}_{\rm per}\rightarrow{\mathcal C}_{\rm per}$, is given by 
	\[
\Delta_h \nu_{i,j} := d_x(D_x \nu)_{i,j} + d_y(D_y \nu)_{i,j} = \frac{1}{h^2}\left( \nu_{i+1,j}+\nu_{i-1,j}+\nu_{i,j+1}+\nu_{i,j-1} - 4\nu_{i,j}\right).
	\]
The 2D vertex-to-center average, $\mathcal{A} : {\mathcal V}_{\rm per}\rightarrow{\mathcal C}_{\rm per}$, is defined  to be 
	\[
\mathcal{A} \nu_{i,j} := \frac{1}{4}\left( \nu_{i+1,j}+\nu_{i-1,j}+\nu_{i,j+1}+\nu_{i,j-1}\right).
	\]		
The 2D \emph{skew} Laplacian, $\Delta^\msfv_h : {\mathcal C}_{\rm per}\rightarrow{\mathcal C}_{\rm per}$, is defined as
	\begin{eqnarray*}
\Delta^\msfv_{h}\nu_{i,j} &=& \md_x(\mD_x\nu)_{i,j} + \md_y( \mD_y\nu)_{i,j} 
	\nonumber
	\\
&=& \frac{1}{2h^2}\left(\nu_{i+1,j+1}+\nu_{i-1,j+1}+\nu_{i+1,j-1}+\nu_{i-1,j-1} - 4\nu_{i,j}\right) .
	\end{eqnarray*}	
	

For $p\ge 2$, the 2D discrete p-Laplacian operator is defined as  
	\begin{eqnarray*}
\nabla_h^\msfv \cdot \left( \left| \nabla_h^\msfv \nu\right|^{p-2} \nabla_h^\msfv \nu \right)_{ij} := \md_x(r\, \mD_x\nu )_{i,j}+ \md_y(r\, \mD_y\nu )_{i,j},
	\end{eqnarray*}	
with 
	\[
r_{i+\frac{1}{2},j+\frac{1}{2}}:=\left[(\mD_x u )_{i+\frac{1}{2},j+\frac{1}{2}}^2+(\mD_y u )_{i+\frac{1}{2},j+\frac{1}{2}}^2\right]^{\frac{p-2}{2}}.
	\]	
Clearly, for $p=2$, $\Delta^\msfv_{h}\nu = \nabla_h^\msfv \cdot \left( \left| \nabla_h^\msfv \nu\right|^{p-2} \nabla_h^\msfv \nu \right)$.
	
Now we are ready to introduce the following grid inner products:  
	\begin{eqnarray*}
\ciptwo{\nu}{\xi}_2 &:=& h^2\sum_{i=1}^m\sum_{j=1}^n \nu_{i,j}\psi_{i,j},\quad \nu,\, \xi\in {\mathcal C}_{\rm per},
\\
\viptwo{\nu}{\xi} &:=& \ciptwo{\mathcal{A}(\nu\xi)}{1}_2 ,\quad \nu,\, \xi\in{\mathcal V}_{\rm per},
\\
\eipew{\nu}{\xi} &:=& \ciptwo{A_x(\nu\xi)}{1}_2 ,\quad \nu,\, \xi\in{\mathcal E}^{\rm ew}_{\rm per},
\\
\eipns{\nu}{\xi} &:=& \ciptwo{A_y(\nu\xi)}{1}_2 ,\quad \nu,\, \xi\in{\mathcal E}^{\rm ns}_{\rm per}.
	\end{eqnarray*}	

We now define the following norms for cell-centered functions. If $\nu\in {\mathcal C}_{\rm per}$, then $\nrm{\nu}_2^2 := \ciptwo{\nu}{\nu}_{2}$; $\nrm{\nu}_p^p := \ciptwo{|\nu|^p}{1}_{2}$ ($1\le p< \infty$), and $\nrm{\nu}_\infty := \max_{1\le i\le m \atop 1\le j\le n}\left|\nu_{i,j}\right|$.
Similarly, we define the gradient norms: for $\nu\in{\mathcal C}_{\rm per}$,
	\[
\nrm{\nabh^\msfv\nu}_p^p := \langle |\nabla_h^\msfv\nu|^p, 1\rangle, \quad |\nabh^\msfv\nu|^p:=[(\mD_x\nu)^2 +(\mD_y\nu)^2]^{\frac{p}{2}} = \left[\nabla_h^\msfv\nu\cdot\nabla_h^\msfv\nu  \right]^{\frac{p}{2}}  \in \mathcal{V}_{\rm per}, \quad 2\le p < \infty, 
	\]
and
	\[
\nrm{ \nabla_h \nu}_2^2 : = \eipew{D_x\nu}{D_x\nu} + \eipns{D_y\nu}{D_y\nu} .
	\]
Consequently, the discrete $\nrm{ \, \cdot \, }_{H_{h}^1}$ and $\nrm { \, \cdot \, }_{H_{h}^2}$  norms on periodic boundary domain defined as 
	\begin{eqnarray}
\nrm{ \phi }_{H_{h}^1}^2 &:=& \nrm{ \phi }_2^2 + \nrm{ \nabla_h \phi }_2^2 ,
	\label{discrete-H1-norm} 
	\\
\nrm{ \phi }_{H_{h}^2}^2 &:=& \nrm{ \phi }_{H_{h}^1}^2 + \nrm{ \Delta_h \phi }_2^2.
	\label{discrete-H2-norm} 
	\end{eqnarray}	

\begin{lem}\label{lem:weak}
For any $\phi \in {\mathcal C}_{\rm per}$, we have 
\begin{equation} 
   \| \nabla_h \phi  \|_2^2 \ge \| \nabla^\msfv_h \phi \|_2^2. 
\end{equation}   
\end{lem}
\begin{proof}
  By the definition of $\mD_x \phi$, we get 
\begin{eqnarray} 
  \mD_x \phi_{i+\hf,j+\hf} = \frac12 \left( (D_x \phi)_{i+\hf,j} + (D_x \phi)_{i+\hf,j+1} \right) , 
  \label{lemma 2.1-1} 
\end{eqnarray} 
which in turn implies that 
\begin{eqnarray} 
  \| \mD_x \phi \|_2^2 := h^2 \sum_{i,j=0}^{m-1}  ( \mD_x \phi_{i+\hf,j+\hf} )^2 
  \le h^2 \sum_{i,j=0}^{m-1}  ( D_x \phi_{i+\hf,j} )^2, 
\end{eqnarray}   
i.e.
\begin{eqnarray} 
\| \mD_x \phi \|_2 \le  \| D_x \phi \|_2.  \label{lemma 2.1-2} 
\end{eqnarray}    
Likewise,  we can also obtain $\| \mD_y \phi \|_2 \le  \| D_y \phi \|_2$. These two inequalities lead to the desired estimate; the proof of Lemma.~\ref{lem:weak} is complete. 
\end{proof}
	
The following preliminary estimates are needed in the convergence analysis presented in later sections; the detailed proof is left to Appendix~\ref{appen:A}.

\begin{prop}\label{prop:1}
For any $\phi \in {\mathcal C}_{\rm per}$ with $\overline{\phi}=0$, we have 
\begin{eqnarray}
  &&
 \nrm{\Delta_h \phi}_2^2  \ge C_1 \| \phi \|_{H_{h}^2}^2 ,  \label{prop 1-0-1} 
\\
  &&
 \nrm{\phi}_\infty \le C \| \phi \|_{H_{h}^2}  ,   \label{prop 1-0-2} 
\\
  &&  
 \nrm{\phi}_{W_h^{1,6}} := \| \phi \|_6 + \| \nabla_h^\msfv \phi \|_6 \le C \| \phi \|_{H_{h}^2} ,  \label{prop 1-0-3} 
\end{eqnarray}
with $C$ and $C_1$ only dependent on $\Omega$. 
\end{prop}


	\subsection{The fully discrete scheme}
	\label{sec:scheme}
Let $M \in\mathbb{Z}^+$, and set $\dt:=T/M$, where $T$ is the final time. We define the canonical grid projection operator $\P: C^0(\Omega)\to {\mathcal C}_{\rm per}$ via $[\P v]_{i,j}=v(\xi_i,\xi_j)$. Set $u_{h,s}:=\P u(\cdot, \dt)$. Then $F_h(u_{h,\dt}) + \frac12 \| \nabla_h ( u_{h,\dt} - u_{h,0} ) \|_2^2 \to F(u(\cdot, 0))$ as $h\to 0$ and $s\to 0$ for sufficiently regular $u$. We denote $\phi_{e}$ as the exact solution to the SS equation \eqref{equation-SS} and take $\Phi_{i,j}^\ell = \P\phi_e (\cdot, t_\ell)$. In the rest of paper, we shall drop the subscription $i,j$ if no confusion is caused.

With the machinery in last subsection, our second-order-in-time BDF type scheme can be formulated as follows:  for $k\geq 1$, given $\phi^{k-1},\phi^k\in {\mathcal C}_{\rm per}$, find $\phi^{k+1}\in {\mathcal C}_{\rm per}$ such that
\begin{eqnarray} 
   \frac{3 \phi^{k+1} - 4 \phi^k + \phi^{k-1}}{2 \dt} 
  &=& \nabla_h^\msfv \cdot ( | \nabla_h^\msfv \phi^{k+1} |^2 \nabla_h^\msfv \phi^{k+1} ) -  \Delta_h^\msfv ( 2 \phi^k - \phi^{k-1} )  \nonumber 
\\
  && 
   - A \dt \Delta_h^2 (\phi^{k+1} - \phi^k)  - \varepsilon^2 \Delta_h^2 \phi^{k+1},
    \label{scheme-BDF} 
\end{eqnarray}
where $\phi^{0}:=\Phi^0$, $\phi^1 := \Phi^1$ and $A$ is the constant stability coefficient. 


For the SS equation \eqref{equation-SS}, we see that the PDE is equivalent if a fixed constant is added or subtracted from the solution. Similar argument could also be applied to the numerical scheme \eqref{scheme-BDF}, since this scheme is mass conservative at a discrete level. For simplicity of presentation, we assume that $\overline{\phi^0} = \overline{\phi^1} = 0$, so that $\overline{\phi^k} = 0$, for any $k \ge 2$.

We now introduce a discrete energy that is consistent with the continuous space energy \eqref{energy-SS-1} as $h\to 0$. In particular, the discrete energy $F_h: {\mathcal C}_{\rm per} \to \mathbb{R}$ is defined as:
 \begin{eqnarray} 
   F_h (\phi) = \frac{1}{4} \nrm{\nabla_h^\msfv\phi}_4^4 
   - \frac12 \nrm{ \nabla_h^\msfv \phi }_2^2+\frac12 \varepsilon^2 \nrm{\Delta_h\phi }_2^2.
   \label{dis-energy-SS} 
 \end{eqnarray} 
\begin{rmk}
We note that $\nrm{\nabla^\msfv_h \phi}_p =0$ does not imply that $\phi$ is a constant. (A checkerboard function has norm zero.) This defect of the skew stencil is not a concern in the present context since the highest order norm in the energy uses a standard stencil.
\end{rmk}

We also denote a modified numerical energy $\tilde{F}_h:{\mathcal C}_{\rm per} \to \mathbb{R} $ via 
 \begin{eqnarray} 
   \tilde{F}_h (\phi, \psi) := F_h (\phi) 
   + \frac{1}{4 \dt} \nrm{ \phi - \psi }_2^2 
   + \frac12 \nrm{ \nabla_h ( \phi - \psi ) }_2^2.  
   \label{mod-energy-SS-1} 
 \end{eqnarray} 
Although we can not guarantee that the energy $F_h$ is non-increasing in time,  we are able to prove the dissipation of auxiliary energy $\tilde{F}_h$.  The unique solvability  and 
the unconditional energy stability of scheme \eqref{scheme-BDF}  is assured by the following theorem.

\begin{thm}
Suppose that the exact solution $\phi_e$ is periodic and sufficiently regular, and $\phi^0, \phi^1 \in {\mathcal C}_{\rm per}$ is obtained via grid projection, as defined above.  Given any $(\phi^{k-1}, \phi^k) \in {\mathcal C}_{\rm per}$, there is a unique solution $\phi^{k+1} \in {\mathcal C}_{\rm per}$ to the scheme \eqref{scheme-BDF}. And also, the scheme \eqref{scheme-BDF}, with starting values $\phi^{0}$ and $\phi^1$, is unconditionally energy stable, \emph{i.e.}, for any $\tau > 0$ and $h>0$, and any positive integer $2\le k \le M-1$, The numerical scheme \eqref{scheme-BDF} has the following energy-decay property: 
\begin{eqnarray} 
    \tilde{F}_h (\phi^{k+1}, \phi^k) \le \tilde{F}_h(\phi^k, \phi^{k-1})\le   \tilde{F}_h(\phi^1, \phi^{0})\leq C_0, 
   \label{scheme-BDF-stability}    
\end{eqnarray}
for all $A\geq \frac{1}{16}$, where $C_0>$ is a constant independent of $s$, $h$ and $T$.
\end{thm}
\begin{proof}
The unique solvability follows from the convexity argument.
Taking an inner product with (\ref{scheme-BDF}) by 
$\phi^{k+1} - \phi^k$ yields
\begin{eqnarray} 
 0&=& \left(\frac{3 \phi^{k+1} - 4 \phi^k + \phi^{k-1}}{2 \dt}, \phi^{k+1} - \phi^k\right) \nonumber \\
 && -\bigg( \nabla_h^\msfv \cdot ( | \nabla_h^\msfv \phi^{k+1} |^2 \nabla_h^\msfv \phi^{k+1} ),\phi^{k+1} - \phi^k \bigg)+ \bigg( \Delta_h^\msfv ( 2 \phi^k - \phi^{k-1} ),\phi^{k+1} - \phi^k \bigg) \nonumber \\ 
   && + A \dt \bigg(\Delta_h^2 (\phi^{k+1} - \phi^k),\phi^{k+1} - \phi^k \bigg)  + \varepsilon^2 \bigg(\Delta_h^2 \phi^{k+1},\phi^{k+1} - \phi^k
    \bigg) \nonumber \\
 &:=& I_1+I_2+I_3+I_4+I_5 .  
    \label{scheme-BDF-inner} 
\end{eqnarray} 
We now establish the estimates for $I_1, \cdots, I_5$. The temporal difference term could be evaluated as follows
\begin{eqnarray}
  \left(  \frac{3 \phi^{k+1} - 4 \phi^k + \phi^{k-1}}{2 \dt} , 
  \phi^{k+1} - \phi^k  \right)  
  \ge \frac{1}{\dt}  \left(  \frac54 \nrm{ \phi^{k+1} - \phi^k }_2^2 
  - \frac14 \nrm{ \phi^k - \phi^{k-1} }_2^2  \right). 
    \label{scheme-BDF-stability-1} 
\end{eqnarray}
For the $4$-Laplacian term, we have 
\begin{eqnarray}
   \left(  - \nabla_h^\msfv \cdot ( | \nabla_h^\msfv \phi^{k+1} |^2 \nabla_h^\msfv \phi^{k+1} )  , 
  \phi^{k+1} - \phi^k  \right)  
  &=& \left(  | \nabla_h^\msfv \phi^{k+1} |^2 \nabla_h^\msfv \phi^{k+1} ,  
  \nabla_h^\msfv ( \phi^{k+1} - \phi^k  )  \right)   \nonumber 
\\ 
  &\ge& \frac14 \left(\| \nabla_h^\msfv \phi^{k+1} \|_{4}^4 
  - \| \nabla_h^\msfv \phi^k \|_{4}^4 \right).
    \label{scheme-BDF-stability-2} 
\end{eqnarray}


For the concave diffusive term, the following estimate is valid
\begin{eqnarray}
 &&\left(  \Delta_h^\msfv ( 2 \phi^k - \phi^{k-1})  , \phi^{k+1} - \phi^k \right)  
  = - \left(  \nabla_h^\msfv ( 2 \phi^k - \phi^{k-1}  ) , 
  \nabla_h^\msfv ( \phi^{k+1} - \phi^k)  \right)   \nonumber 
\\
&=&- \left(  \nabla_h^\msfv \phi^k , 
  \nabla_h^\msfv ( \phi^{k+1} - \phi^k)  \right)  
  - \left(  \nabla_h^\msfv (\phi^k - \phi^{k-1}  ) , 
  \nabla_h^\msfv ( \phi^{k+1} - \phi^k)  \right)   \nonumber
\\
&=&-\frac12\| \nabla^\msfv_h \phi^{k+1} \|_2^2 +\frac12\| \nabla^\msfv_h \phi^{k} \|_2^2 +\frac12\| \nabla^\msfv_h (\phi^{k+1}-\phi^k) \|_2^2 
  - \left(  \nabla_h^\msfv (\phi^k - \phi^{k-1}  ) , 
  \nabla_h^\msfv ( \phi^{k+1} - \phi^k)  \right)   \nonumber 
\\
 &\ge& - \frac12 \left(  \| \nabla^\msfv_h \phi^{k+1} \|_2^2 
  - \| \nabla^\msfv_h \phi^k \|_2^2   \right)  
  - \frac12 \| \nabla^\msfv_h ( \phi^k - \phi^{k-1} ) \|_2^2 \nonumber 
  \\
 &\ge& - \frac12 \left(  \| \nabla^\msfv_h \phi^{k+1} \|_2^2 
  - \| \nabla^\msfv_h \phi^k \|_2^2   \right)  
  - \frac12 \| \nabla_h ( \phi^k - \phi^{k-1} ) \|_2^2,
    \label{scheme-BDF-stability-5}  
\end{eqnarray}
where the last step applied the Lemma~\ref{lem:weak}.

%

For the surface diffusion term, we have 
\begin{eqnarray} 
  \left(  \Delta_h^2 \phi^{k+1}  ,  \phi^{k+1} - \phi^k \right)  
  =  \left(  \Delta_h \phi^{k+1}  , 
   \Delta_h (\phi^{k+1} - \phi^k)  \right) 
  \ge \frac12 \left(  \| \Delta_h \phi^{k+1} \|_2^2 
  - \| \Delta_h \phi^k \|_2^2   \right) .
    \label{scheme-BDF-stability-3}  
\end{eqnarray} 
Similarly, the following identity is valid for the stabilizing term:  
\begin{eqnarray} 
  s \left(  \Delta_h^2 ( \phi^{k+1} - \phi^k ) , \phi^{k+1} - \phi^k  \right)  
  =  s \| \Delta_h ( \phi^{k+1} - \phi^k ) \|_2^2. 
    \label{scheme-BDF-stability-4}  
\end{eqnarray} 
Meanwhile, an application of Cauchy inequality indicates the following estimate: 
\begin{eqnarray} 
   \frac{1}{\dt} \| \phi^{k+1} - \phi^k \|_2^2
   + A \dt \| \Delta_h ( \phi^{k+1} - \phi^k ) \|_2^2  
   \ge 2 A^{1/2}  \| \nabla_h ( \phi^{k+1} - \phi^k ) \|_2^2 . 
  \label{scheme-BDF-stability-6}    
\end{eqnarray}


Therefore, a combination of (\ref{scheme-BDF-stability-1})-(\ref{scheme-BDF-stability-5}) and (\ref{scheme-BDF-stability-6}) yields 
\begin{eqnarray} 
  &&
  F_h (\phi^{k+1}) - F_h (\phi^k) 
  + \frac{1}{4 \dt}  \left(  
   \| \phi^{k+1} - \phi^k \|_2^2 - \| \phi^k - \phi^{k-1} \|_2^2  \right)  \nonumber 
\\
  &&   
  + \frac12 \left( \| \nabla_h ( \phi^{k+1} - \phi^k ) \|_2^2 
  - \| \nabla_h ( \phi^k - \phi^{k-1} ) \|_2^2  \right) \nonumber \\
  &&\le ( - 2 A^{1/2} + \frac12 ) \| \nabla_h ( \phi^{k+1} - \phi^k ) \|_2^2  \le 0 ,  
  \label{scheme-BDF-stability-7}    
\end{eqnarray} 
provided that $A \ge \frac{1}{16}$. 
Then the proof follows from the definition of the $\tilde{F}_h$ in \eqref{mod-energy-SS-1}.
\end{proof}


 \subsection{$L^\infty_h (0,T; H_h^2)$ Stability of the Numerical Scheme} \label{sec:h3stability}  
    The $L_h^\infty(0,T; H_{h}^2)$ bound of the numerical solution could be derived based on the modified energy stability (\ref{scheme-BDF-stability}). 
\begin{thm}
Let $\phi \in {\mathcal{C}}_\Omega$, then the $L^\infty_h(0,T; H_{h}^2)$ bound of the numerical solution is as follows:
    \begin{equation} 
     \nrm{\phi}_{H_{h}^2}\leq \sqrt{2 \frac{C_0+| \Omega |}{C_1\varepsilon^{2}}}
        := C_2 ,  \label{leading stability} 
    \end{equation} 
where $C_2$ is independent of $s$, $h$ and $T$.     
\end{thm}    
\begin{proof}
Since  
    \begin{equation} 
      \frac18 \psi^4 - \frac12 \psi^2 \ge -\frac12 ,
   \end{equation} 
then we have      
	\begin{equation} 
      \frac18 \|\nabla^\msfv_h\phi \|_4^4 - \frac12 \|\nabla_h^\msfv \phi \|_2^2 \ge - \frac12 | \Omega | ,  
	  \label{H1 bound-1} 
	\end{equation}  
with the discrete $H_h^1$ norm introduced in \eqref{discrete-H1-norm}. 
   Then we arrive at the following bound, for any $\phi \in {\mathcal{C}}_\Omega$:  
   \begin{eqnarray} 
     F_h (\phi) &\ge& \frac18 \|\nabla_h^\msfv \phi \|_4^4 + \frac{\varepsilon^2}{2} \| \Delta_h \phi \|_2^2 - \frac12 | \Omega | \nonumber\\
     &\ge& \frac12 \| \nabla_h^\msfv\phi \|_2^2 + \frac{\varepsilon^2}{2} \| \Delta_h \phi \|_2^2 - | \Omega | \nonumber \\
  & \ge&  \frac{\varepsilon^2}{2} \| \Delta_h \phi \|_2^2 - | \Omega | \nonumber\\
     &\ge& \frac12 C_1\varepsilon^2 \| \phi \|_{H_{h}^2}^2 - | \Omega | ,
 \end{eqnarray}  
in which $C_1$ is a constant associated with the discrete elliptic regularity: $\| \Delta_h \phi \|_2^2 \ge C_1 \| \phi \|_{H_{h}^2}^2$, as stated in \eqref{prop 1-0-1} of Proposition~\ref{prop:1}.  Consequently, its combination with \eqref{mod-energy-SS-1} finishes the proof.
\end{proof}

	\begin{rem}
Note that the constant $C_2$ is independent of $s$, $h$ and $T$, but does depends on $\varepsilon$. In particular, $C_2=O (\varepsilon^{-1})$.
	\end{rem}
    
\section{Convergence Analysis and Error Estimate} \label{sec:convergence}

\subsection{Error equations and consistency analysis} \label{subsec:error-equations}
A detailed Taylor expansion implies the following truncation error: 
\begin{eqnarray} 
\frac{3\Phi^{k+1} -4 \Phi^k+\Phi^{k-1}}{2\dt} &=&  \nabla_h^\msfv \cdot ( | \nabla_h^\msfv \Phi^{k+1} |^2 \nabla_h^\msfv \Phi^{k+1} ) -  \Delta_h^\msfv ( 2 \Phi^k - \Phi^{k-1} )  \nonumber 
\\
 &&- A \dt \Delta_h^2 (\Phi^{k+1} - \Phi^k)  - \varepsilon^2 \Delta_h^2 \Phi^{k+1}
 + \tau^k , \label{consistency-SS} 
\end{eqnarray}
with $\nrm{ \tau^k }_2 \le C (h^2+\dt^2)$ . Consequently, with an introduction of the error function
\begin{eqnarray} 
  e^k = \Phi^k - \phi^k ,   \quad \forall \, k \ge 0 , 
  \label{error-unction-1} 
\end{eqnarray}
we get the following evolutionary equation, by subtracting \eqref{scheme-BDF} from \eqref{consistency-SS}:
\begin{eqnarray} 
\frac{3e^{k+1} -4 e^k+e^{k-1}}{2\dt} &=&  \nabla_h^\msfv \cdot ( | \nabla_h^\msfv \Phi^{k+1} |^2 \nabla_h^\msfv \Phi^{k+1}  -  | \nabla_h^\msfv \phi^{k+1} |^2 \nabla_h^\msfv \phi^{k+1} ) \nonumber 
\\
 &&- \Delta_h^\msfv ( 2 e^k - e^{k-1} )  - A \dt \Delta_h^2 (e^{k+1} - e^k)  - \varepsilon^2 \Delta_h^2 e^{k+1}
 + \tau^k , \label{consistency-SS-2} 
\end{eqnarray}
In addition, from the PDE analysis for the SS equation in  \cite{li03,li04} and the global in time $H_{h}^2$ stability \eqref{leading stability} for the numerical solution,  we also get the $L_h^\infty$, $W^{1,6}$ and $H_{h}^2$ bounds for both the exact solution and numerical solution, uniform in time: 
	\begin{equation} 
\| \Phi^k \|_{\infty } ,  \  \| \Phi^k \|_{W^{1,6} }  , \  \| \Phi^k \|_{H_{h}^2 }  \le C_3 ,  \quad  \| \phi^k \|_{\infty } ,  \  \| \phi^k \|_{W^{1,6} }  , \  \| \phi^k \|_{H_{h}^2 }  \le C_3 ,  \quad  \forall \, k \ge 0 ,
	\label{consistency-bound-1}
	\end{equation}
where the 3-D embeddings of $H_{h}^2$ into $L_h^\infty$ and into $W^{1, 6}$ have been applied, as well as the discrete Sobolev embedding inequalities \eqref{prop 1-0-2},  \eqref{prop 1-0-3} in Proposition~\ref{prop:1}.
 
\subsubsection{Stability and convergence analysis}
 
The convergence result is stated in the following theorem. 

\begin{thm}\label{thm:convergence-2nd}
Let $\Phi \in \mathcal{R}$  be the projection of the exact periodic solution of the SS equation \eqref{equation-SS} with the initial data $\phi^{0}:=\Phi^0 \in H^2_{\rm per}(\Omega)$, $\phi^1 := \Phi^1\in H^2_{\rm per}(\Omega)$, and the regularity class
\begin{eqnarray} 
  \mathcal{R} = H^3 (0,T; C^0(\Omega)) \cap H^2 (0,T; C^2(\Omega)) \cap H^1 (0,T; C^4(\Omega)) 
  \cap L^\infty (0,T; C^6(\Omega))  . 
  \label{regularity assumption} 
\end{eqnarray}    
Suppose $\phi$ is the fully-discrete solution of \eqref{scheme-BDF}. Then the following convergence result holds as $s$, $h$ goes to zero: 
\begin{eqnarray} 
  \| e^k \|_2  
  + \left( \frac{3}{16}\varepsilon^2s \sum_{\ell=0}^{k} \| \Delta_h  e^\ell  \|^2  \right)^{1/2} 
  \le C  ( s^2 + h^2 )  ,  
\label{thm:convergence-FD-2nd}
\end{eqnarray}     
where the constant $C>0$ is independent of $s$ and $h$. 
\end{thm}
\begin{proof}
Taking an  inner product with the numerical error equation \eqref{consistency-SS-2} by $e^{k+1}$ gives
 \begin{eqnarray} 
0&=&\left(\frac{3e^{k+1} -4 e^k+e^{k-1}}{2\dt}, e^{k+1}\right)+\left( | \nabla_h^\msfv \Phi^{k+1} |^2 \nabla_h^\msfv \Phi^{k+1}-| \nabla_h^\msfv \phi^{k+1} |^2 \nabla_h^\msfv \phi^{k+1}, \nabla_h^\msfv e^{k+1}\right)\nonumber \\
&& -\left(\nabla_h^\msfv ( 2 e^k - e^{k-1} ), \nabla_h^\msfv e^{k+1}\right)  
+A \dt\left(\Delta_h (e^{k+1} - e^k),\Delta_h e^{k+1}\right)\nonumber \\
&& +\varepsilon^2 \left(\Delta_h e^{k+1}, \Delta_h e^{k+1} \right)
 -\left( \tau^k,  e^{k+1}\right)\nonumber \\
&=:& J_1+ J_2+J_3+J_4+J_5+J_6. \label{consistency-SS-3} 
 \end{eqnarray}
For the time difference  error term $J_1$,
  \begin{eqnarray} 
\left(\frac{3e^{k+1} -4 e^k+e^{k-1}}{2\dt}, e^{k+1}\right)&=&\frac{3}{4\dt} \| e^{k+1} \|_2^2-\frac{1}{\dt} \| e^{k} \|_2^2 + \frac{1}{4\dt} \| e^{k-1} \|_2^2 \nonumber\\
&&+\frac{1}{\dt} \| e^{k+1}-e^k \|_2^2-\frac{1}{4\dt} \| e^{k+1}-e^{k-1} \|_2^2.
\label{error-J1}
  \end{eqnarray}
For the backwards diffusive error term $J_3$, we have
\begin{eqnarray}
 -\left(\nabla_h^\msfv  ( 2 e^k - e^{k-1} ), \nabla_h ^\msfv e^{k+1}\right)&=& -\frac12\| \nabla_h^\msfv  e^{k+1} \|_2^2 -  \|\nabla_h^\msfv  e^{k} \|_2^2 + \frac12\| \nabla_h^\msfv  e^{k-1} \|_2^2\nonumber\\
 && +\| \nabla_h^\msfv  (e^{k+1}-e^k) \|_2^2 -\frac12 \| \nabla_h^\msfv  (e^{k+1}-e^{k-1}) \|_2^2.
\label{error-J3}  
\end{eqnarray} 
And for the stabilizing term $J_4$, 
\begin{eqnarray} 
A \dt\left(\Delta_h (e^{k+1} - e^k),\Delta_h e^{k+1}\right)&=& \frac{A\dt}{2}\left( \| \Delta_h e^{k+1} \|_2^2 - \| \Delta_h e^{k} \|_2^2 + \| \Delta_h (e^{k+1}-e^{k}) \|_2^2 \right). 
 \label{error-J4}  
\end{eqnarray}  
For the surface diffusion error term $J_5$ and the local truncation error term $J_6$, we have 
\begin{eqnarray} 
\varepsilon^2 \left(\Delta_h e^{k+1}, \Delta_h e^{k+1} \right)=\varepsilon^2 \| \Delta_h e^{k+1} \|_2^2,
\label{error-J5}  
\end{eqnarray}
and
\begin{eqnarray} 
 -\left( \tau^k,  e^{k+1}\right) \leq \| \tau^k \|_2 \cdot \| e^{k+1} \|_2 \leq  \frac12 \| \tau^k \|_2^2+\frac12 \| e^{k+1} \|_2^2.
\label{error-J6}  
\end{eqnarray}
For the nonlinear error term $J_2$, we adopt the same trick in \cite{feng2016fch}, and get   
\begin{eqnarray} 
J_2&=&\left( | \nabla_h^\msfv  \Phi^{k+1} |^2 \nabla_h^\msfv  \Phi^{k+1}-| \nabla_h^\msfv  \phi^{k+1} |^2 \nabla_h^\msfv  \phi^{k+1}, \nabla_h^\msfv  e^{k+1}\right)\nonumber\\
&=&\left( \nabla_h^\msfv  (\Phi^{k+1}+\phi^{k+1}) \cdot\nabla_h^\msfv  e^{k+1}\nabla_h^\msfv  \Phi^{k+1}, \nabla_h^\msfv  e^{k+1}\right)+\left( | \nabla_h^\msfv  \phi^{k+1} |^2 \nabla_h^\msfv  e^{k+1}, \nabla_h^\msfv  e^{k+1}\right)\nonumber\\
&=:& J_{2,1}+J_{2,2}.
\label{error-J2}  
\end{eqnarray}
For the first part $J_{2,1}$ of \eqref{error-J2}, we have 
\begin{eqnarray} 
 - J_{2,1}&\leq&C_4 \left( \| \nabla_h^\msfv  \Phi^{k+1} \|_6 + \| \nabla_h^\msfv  \phi^{k+1} \|_6 \right)\cdot \| \nabla_h^\msfv  \Phi^{k+1} \|_6 \cdot \| \nabla_h^\msfv  e^{k+1} \|_6 \cdot \| \nabla_h^\msfv  e^{k+1} \|_2\nonumber\\
&\leq& C_5C_3^2 \| \nabla_h^\msfv  e^{k+1} \|_6 \cdot \| \nabla_h^\msfv  e^{k+1} \|_2\nonumber\\
&\leq& C_5C_3^2 \| \nabla_h  e^{k+1} \|_6 \cdot \| \nabla_h  e^{k+1} \|_2\nonumber\\
&\leq&C_6 \| \Delta_h e^{k+1} \|_2 \cdot \| e^{k+1} \|_2^{\frac{1}{2}} \| \Delta_h  e^{k+1} \|_2^{\frac{1}{2}}\nonumber\\
&\leq&C_7 \| e^{k+1} \|_2^{\frac{1}{2}} \cdot \| \Delta_h  e^{k+1} \|_2^{\frac{3}{2}} \nonumber\\
&\leq&C_8 \| e^{k+1} \|_2^{2} +\frac{3}{4}\varepsilon^2 \| \Delta_h e^{k+1} \|_2^{2} , 
\label{error-J21}  
\end{eqnarray}
in which the $W^{1,6}$ bound \eqref{consistency-bound-1} for the exact and numerical solutions was recalled in the second step, the Sobolev embedding from $H_{h}^2$ into $W^{1,6}$ and the estimate \eqref{consistency-bound-1} were used in the last step.  The estimate for the second part $J_{2,2}$ of \eqref{error-J2} is trivial:  
\begin{eqnarray} 
 J_{2,2} \ge 0 . 
\label{error-J22}  
\end{eqnarray}
Then we arrive at
\begin{eqnarray} 
- J_2&\leq&C_{9} \| e^{k+1} \|_2^{2}+\frac{3}{4}\varepsilon^2 \| \Delta_h e^{k+1} \|_2^{2}.
\label{error-J2-1}  
\end{eqnarray}
Finally, a combination of  \eqref{error-J1}, \eqref{error-J3}, \eqref{error-J4}, \eqref{error-J5}, \eqref{error-J6} and \eqref{error-J2-1}  yields that
  \begin{eqnarray} 
&&\frac{3}{4\dt}\left( \| e^{k+1} \|_2^2 - \| e^{k} \|_2^2 \right)-\frac{1}{4\dt}\left( \| e^{k} \|_2^2 - \| e^{k-1} \|_2^2\right)+\frac{1}{2\dt} \| e^{k+1}-e^{k} \|_2^2\nonumber\\
&&-\frac{1}{2\dt} \| e^{k}-e^{k-1} \|_2^2 + \frac{As}{2}\left( \| \Delta_h e^{k+1} \|_2^2 - \| \Delta_h e^{k} \|_2^2\right) + \varepsilon^2 \| \Delta_h e^{k+1} \|_2^2\nonumber\\
&\leq& \frac12 \| \tau^k \|_2^2 + \frac12 \| e^{k+1} \|_2^2 +C_{9}  \| e^{k+1} \|_2^{2} + \frac{3}{4}\varepsilon^2 \| \Delta_h e^{k+1} \|_2^{2}\\
&&-\| \nabla_h^\msfv e^{k+1} \|_2^2 - 2 \| \nabla_h^\msfv e^{k} \|_2^2 - \| \nabla_h^\msfv (e^{k+1}-e^{k-1}) \|_2^2 \nonumber \\
&&+ 4 \varepsilon^{-2} \| e^{k+1} \|_2^2 + 288\varepsilon^{-2} \| e^{k} \|_2^2 + 72 \varepsilon^{-2} \| e^{k-1} \|_2^2   \nonumber\\
&&+ \frac{1}{16}\varepsilon^2 \left( \| \Delta_h e^{k+1} \|_2^2 + \| \Delta_h e^{k} \|_2^2 + \|\Delta_h e^{k-1} \|_2^2 \right)\nonumber.
\label{error-combination}
  \end{eqnarray}
A summation in time implies that
  \begin{eqnarray} 
&&\frac{3}{4\dt}\left( \| e^{k+1} \|_2^2 - \| e^{1} \|_2^2\right)-\frac{1}{4\dt}\left( \| e^{k} \|_2^2 - \| e^{0} \|_2^2 \right)+\frac{1}{2\dt} \| e^{k+1}-e^{k} \|_2^2\nonumber\\
&&-\frac{1}{2\dt} \| e^{1}-e^{0} \|_2^2 + \frac{As}{2}\left( \| \Delta_h e^{k+1} \|_2^2 - \| \Delta_h e^{0} \|_2^2 \right) + \frac{3}{16} \varepsilon^2 \sum _{\ell=1}^{k} \| \Delta_h e^{\ell+1} \|_2^2 \nonumber\\
&\leq& \frac12\sum _{\ell=1}^{n} \| \tau^\ell \|_2^2 + \sum _{\ell=1}^{k} \left(\frac12 +C_{9} + 4 \varepsilon^{-2} \right) \| e^{\ell+1} \|_2^{2} \\
&&+72\varepsilon^{-2}\sum _{\ell=1}^{k}\left(4 \| e^{\ell} \|_2^2 + \| e^{\ell-1} \|_2^2\right)+ \frac{1}{16}\varepsilon^2 \sum_{\ell=1}^{k} \left( \| \Delta_h e^{\ell} \|_2^2 + \| \Delta_h e^{\ell-1} \|_2^2 \right) . \nonumber
\label{sum}
  \end{eqnarray}
  In turn, an application of discrete Gronwall inequality yields the desired convergence result \eqref{thm:convergence-2nd}. This completes the proof of Theorem~\ref{thm:convergence-FD-2nd}. 
\end{proof}
\section{Precondition Steepest Descent Solver}\label{sec:psd}
In this section we describe a preconditioned steepest descent (PSD) algorithm following the practical and  theoretical framework in~\cite{feng2016preconditioned}. The fully discrete scheme \eqref{scheme-BDF} can be recast as a minimization problem: For any $\phi \in \mathcal{C}_{\rm per}$, the following energy functional is introduced: 
	\begin{eqnarray}
E_h[\phi]&=& \frac{3}{\dt}\nrm{\phi}_2^2+\frac{1}{4} \nrm{\nabla_h^\msfv \phi}_4^4 +\frac12 (A \dt+\varepsilon^2 )\nrm{\Delta_h\phi }_2^2 . \label{equation-energy}
	\end{eqnarray}
 One observes that the fully discrete scheme \eqref{scheme-BDF} is the discrete variation of the strictly convex energy \eqref{equation-energy} set equal to zero. The nonlinear scheme at a fixed time level may be expressed as
	\begin{equation}
\mathcal{N}_h[\phi] = f,
	\end{equation}
with 
	\begin{eqnarray} 
\mathcal{N}_h[\phi] =    \frac{3}{2 }\phi^{k+1} -\dt\nabla_h^\msfv  \cdot ( | \nabla_h^\msfv  \phi^{k+1} |^2 \nabla_h^\msfv  \phi^{k+1} ) +(A \dt^2+\dt\varepsilon^2 ) \Delta_h^2\phi^{k+1} , \label{non-operator}
	\end{eqnarray} 
and
	\begin{eqnarray} 
f =  \frac{1}{2}(4 \phi^k - \phi^{k-1})
   - \dt \Delta_h^\msfv ( 2 \phi^k - \phi^{k-1} )  + A\dt^2 \Delta_h^2 \phi^k  .
	\end{eqnarray} 	
The main idea of the PSD solver is to use a linearized version of the nonlinear operator as a pre-conditioner, or in other words, as a metric for choosing the search direction.  A linearized version of the nonlinear operator $\mathcal{N}$, denoted as $\mathcal{L}_h: \mathring{\mathcal C}_{\rm per} \to \mathring{\mathcal C}_{\rm per}$,  is defined as follows: 	\begin{equation*}
{\mathcal L}_h[\psi] :=  \frac{3}{2 }\psi - \dt\Delta_h \psi  +(A \dt^2+\dt\varepsilon^2 ) \Delta_h^2 \psi.
	\end{equation*}
Clearly, this is a positive, symmetric operator, and we use this as a pre-conditioner for the method. Specifically, this ``metric" is used to find an appropriate search direction for the steepest descent solver~\cite{feng2016preconditioned}. Given the current iterate $\phi^{n}\in {\mathcal C}_{\rm per}$, we define the following \emph{search direction} problem: find $d^n \in \mathring{\mathcal C}_{\rm per}$ such that
\[
{\mathcal L}_h[d^n]= f-\mathcal{N}_h[\phi^n]:=r^n,
\]
where $r^n$ is the nonlinear residual of the $n^{\rm th}$ iterate $\phi^n$. This last equation  can be solved efficiently using the Fast Fourier Transform (FFT).

We then obtain the next iterate as
	\begin{equation}
\phi^{n+1} = \phi^{n} + \overline{\alpha} d^n,
	\end{equation}
where $\overline{\alpha}\in\mathbb{R}$ is the unique solution to the steepest descent line minimization problem 
	\begin{equation}
\overline{\alpha} := \operatorname*{argmax}_{\alpha\in\mathbb{R}} E_h[\phi^{n} + \alpha d^n]= \operatorname*{argzero}_{\alpha\in\mathbb{R}}\delta E_h[\phi^{n} + \alpha d^n](d^n) .
	\label{eqn-search}
	\end{equation}
The theoretical analysis in~\cite{feng2016preconditioned} suggests that the iteration sequence $\phi^n$ converges geometrically to $\phi^{k+1}$, with $\phi^{k+1}$ the exact numerical solution of scheme \eqref{scheme-BDF} at time level $k+1$, \emph{i.e.},  $\mathcal{N}_h[\phi^{k+1}] = f$. And also, this analysis implies a convergence rate  independent of $h$. 

	\begin{rem}
	\label{rem: CN-1}
The Crank-Nicolson version of the second order energy stable scheme for the SS equation \eqref{equation-SS}, proposed and analyzed in \cite{shen2012second}, takes the following (spatially-continuous) form:
	\begin{equation}
	\label{scheme-CN-1}
	\begin{aligned}
  \frac{\phi^{k+1}-\phi^k}{s} = &   \chi ( \nabla \phi^{k+1}, \nabla \phi^k) - \Delta \left( \frac32 \phi^k - \frac12 \phi^{k-1} \right) - \frac{\varepsilon^2}{2} \Delta^2 \left( \phi^{k+1}  + \phi^{k} \right) ,
	\\
  \chi ( \nabla \phi^{k+1}, \nabla \phi^k)  := & \frac{1}{4} \nabla \cdot \left( ( | \nabla \phi^{k+1} |^2 + | \nabla \phi^k |^2 )  \nabla ( \phi^{k+1} + \phi^k ) \right).
	\end{aligned}
	\end{equation}
In this numerical approach, every terms in the chemical potential are evaluated at time instant $t^{k+1/2}$. 

Both the CN version \eqref{scheme-CN-1} and the BDF one \eqref{scheme-BDF} require a nonlinear solver, while the nonlinear term in \eqref{scheme-CN-1} takes a more complicated form than \eqref{scheme-BDF}, which comes from different time instant approximations. As a result, a stronger convexity of the nonlinear term in the BDF one \eqref{scheme-BDF} is expected to greatly improve the numerical efficiency in the nonlinear iteration. 

Such a numerical comparison has been undertaken for the Cahn-Hilliard (CH) model in recent works: the CN and BDF versions of second order accurate, energy stable numerical schemes for the CH equation, proposed in~\cite{guo16}, \cite{yan16a}, respectively, were tested using the same numerical set-up. The numerical experiments have indicated that, since the nonlinear term in the BDF approach has a stronger convexity than the one in the CN one, a 20 to 25 percent improvement of the computational efficiency is generally available for the CH model. 

For the numerical comparison between the BDF and CN approaches for the SS equation \eqref{equation-SS}, namely \eqref{scheme-BDF}, \eqref{scheme-CN-1}, respectively. Such an efficiency improvement is expected to be much greater. This expectation comes from a subtle fact that, the modified CN approximation to the 4-Laplacian term, $\chi ( \nabla \phi^{k+1}, \nabla \phi^k)$, does not correspond to a convex energy functional, because of the vector gradient form (other than a scalar form) in the 4-Laplacian expansion. As a consequence, the PSD algorithm proposed in this section could hardly be efficiently applied to solve for \eqref{scheme-CN-1}, while the PSD application to the BDF approach \eqref{scheme-BDF} has led to a great success. In fact, an application of the Polak-Ribi\'ere variant of NCG method \cite{polak69} to solve for \eqref{scheme-CN-1}, as reported in \cite{shen2012second}, has shown a fairly poor numerical performance.  
	\end{rem}

\section{Preconditioned Nonlinear Conjugate Gradient Solvers}\label{sec:pncg}
Based on the PSD algorithm, we define $\overline{g}_k = \mathcal{L}_h^{-1}(r^k)$. Then our PNCG algorithms are given by the following equations:
\begin{eqnarray}
\phi^{k+1} &=& \phi^k + \overline{\alpha}_k d^k\\
d^{k+1} &=& -\overline{g}_{k+1} + \overline{\beta}_{k+1}d^{k}, d^0 = -\overline{g}_{0}. 
\end{eqnarray}

And more details can be found in Algorithm~\ref{algorithm:pncg}. 

 \begin{algorithm}
  \caption{Linearly Preconditioned Nonlinear Conjugate Gradient (PNCG) Method}
  \label{algorithm:pncg}
  \begin{algorithmic}[1]
  \State Compute residual: $r^0:=f-\mathcal{N}_h(\phi^0)$
  \State Set $\overline{g}_{0} = \mathcal{L}_h^{-1}(r^0)$
  \State Set $d^0 \leftarrow -\overline{g}_{0}, k \leftarrow 0$
  \While{$\overline{g}_{k} \neq 0$}
  	\State Compute $\overline{\alpha}_k$ 
 \Comment{secant search}
  \State  $\phi^{k+1} \leftarrow \phi^{k}+\overline{\alpha}_k d^k$ \Comment{steepest descent algorithm}
  	\State $\overline{g}_{k+1} \leftarrow \mathcal{L}_h^{-1}(r^{k+1})
  	       = \mathcal{L}_h^{-1}(f-\mathcal{N}_h(\phi^{k+1})) $
    \State Compute $\overline{\beta}_{k+1}$    
   \State  $d^{k+1} \leftarrow -\overline{g}_{k+1}+\overline{\beta}_{k+1} d^k$   
   \State $k\leftarrow k+1$  
  \EndWhile
  \end{algorithmic}
  \end{algorithm}	

However, there several different ways to choose the scaling parameter $\overline{\beta}_{k+1}$. And two of the best known formulas for $\overline{\beta}_{k+1}$ are named after their deveiops:
\begin{description}[align=left,labelwidth=1cm]
\item \textbf{Fletcher-Reeves \cite{fletcher1964function}:}
\begin{eqnarray}
\overline{\beta}_{k+1}^{FR} =\frac{\overline{g}_{k+1}^T\overline{g}_{k+1}}{\overline{g}_{k}^T\overline{g}_{k}}
\label{eqn:betaFR}
\end{eqnarray}
\item \textbf{Polak-Ribi\`ere \cite{ere1969note}:}
\begin{eqnarray}
\overline{\beta}_{k+1}^{PR} =\frac{\overline{g}_{k+1}^T(\overline{g}_{k+1}-\overline{g}_{k})}{\overline{g}_{k}^T\overline{g}_{k}}
\label{eqn:betaPR}
\end{eqnarray}
\end{description}

Based on those two best known formulas, we proposed the following two PNCG solvers:

\begin{description}[align=left,labelwidth=1cm]
\item \textbf{PNCG1:}
\begin{eqnarray}
\overline{\beta}_{k+1} =\max{\{0,\overline{\beta}_{k+1}^{PR}\}}
\end{eqnarray}
\item \textbf{PNCG2 :}
\begin{eqnarray}
\overline{\beta}_{k+1} =\max{\{0,\min\{\overline{\beta}_{k+1}^{FR},\overline{\beta}_{k+1}^{PR}\}\}}
\end{eqnarray}
\end{description}

\begin{rmk}
The PNCG2 is also called hybrid conjugate gradient algorithm in \cite{zhou2005preconditioned}.
\end{rmk}

\section{Numerical Experiments}\label{sec:num}
	\subsection{Convergence test and the complexity of the Preconditioned solvers}
	\label{subsec-convergence}
In this subsection we demonstrate the accuracy and complexity of the preconditioned solvers. We present the results of the convergence test and perform some sample computations to investigate the effect of the time step $\dt$ and stabilized parameter $A$  for the energy $F_h(\phi)$.  
	
To simultaneously demonstrate the spatial accuracy and the efficiency of the solver, we perform a typical time-space convergence test for the fully discrete scheme \eqref{scheme-BDF} for the slope selection model.
As in \cite{chen16b, shen2012second, wang10a},  we perform the Cauchy-type convergence test using the following periodic initial data~\cite{shen2012second}:
	\begin{eqnarray}
	\label{eqn:init2nd}
 u (x,y,0)&=&0.1\sin^2\left(\frac{2\pi x}{L}\right)\cdot \sin\left(\frac{4\pi (y-1.4)}{L}\right)
 	\nonumber
	\\
&& - 0.1\cos\left(\frac{2\pi (x-2.0)}{L}\right)\cdot\sin\left(\frac{2\pi y}{L}\right),
	\end{eqnarray}
with $\Omega=[0, 3.2]^2$, $\varepsilon = 0.1$, $\dt=0.01h$, $A=\nicefrac{1}{16}$ and $T=0.32$. We use a linear refinement path, \emph{i.e.}, $s=Ch$. At the final time $T=0.32$, we expect the global error to be $\mathcal{O}(s^2)+\mathcal{O}(h^2)=\mathcal{O}(h^2)$,  in either the  $L_h^2$ or $L_h^\infty$ norm, as $h, s\to 0$. The Cauchy difference is defined as $\delta_\phi: =\phi_{h_f}-\mathcal{I}_c^f(\phi_{h_c})$, where $\mathcal{I}_c^f$ is a bilinear interpolation operator (with the Nearest Neighbor Interpolation applied in Matlab, which is similar to the 2D case in \cite{feng2016fch,feng2016preconditioned} and the 3D case in \cite{dong2016convergence}). This requires a relatively coarse solution, parametrized by $h_c$, and a relatively fine solution, parametrized by $h_f$, in particular $h_c = 2 h_f$, at the same final time. The $L_h^2$ norms of Cauchy difference and the convergence rates can be found in Table~\ref{tab:cov-bdf2etf}. The  results confirm our expectation for the second-order convergence in both space and time.

\begin{table}[!htb]
	\begin{center}
		\caption{Errors, convergence rates, average iteration numbers and average CPU time (in seconds) for each time step. Parameters are given in the
		text, and the initial data is defined in \eqref{eqn:init2nd}. The refinement path is $s=0.01h$. } \label{tab:cov-bdf2etf}
		\begin{tabular}{ccccccccccccc}
		\hline &&&&\multicolumn{2}{c}{PSD}&\multicolumn{2}{c}{PNCG1}&\multicolumn{2}{c}{PNCG2}\\
			\hline $h_c$&$h_{f}$&$\nrm{\delta_\phi}_{2}$ & Rate&$\#_{iter}$ &$T_{cpu}(h_f)$&$\#_{iter}$ &$T_{cpu}(h_f)$&$\#_{iter}$ &$T_{cpu}(h_f)$\\
			\hline $\frac{3.2}{16}$&$\frac{3.2}{32}$& $1.3938\times 10^{-2}$&- & 11  & 0.0019
			& 9 & 0.0016 & 9 & 0.0015
			\\$\frac{3.2}{32}$&$\frac{3.2}{64}$& $1.7192\times 10^{-3}$ & 3.02& 10 & 0.0103
			& 9 & 0.0093 & 8 & 0.0085
			\\ $\frac{3.2}{64}$ &$\frac{3.2}{128}$& $3.8734\times 10^{-4}$& 2.15 & 08 & 0.0529
			& 8 & 0.0486 & 7 & 0.0454
			\\ $\frac{3.2}{128}$ & $\frac{3.2}{256}$ &$9.4766\times 10^{-5}$& 2.03 & 07 & 0.2512
			& 7 & 0.2038 & 6 & 0.2046
			\\ $\frac{3.2}{256}$ & $\frac{3.2}{512}$&$2.3564\times 10^{-5}$& 2.01 & 07 & 1.6650
			& 7 & 1.6268 & 6 & 1.5207
			\\
			\hline
		\end{tabular}
	\end{center}
\end{table}

In the second part of this test, we demonstrate the complexity of the preconditioned solvers  with initial data \eqref{eqn:init2nd}. In Figure~\ref{fig:complexitybdf2etf}, we plot the semi-log scale of the relative residuals versus preconditioned solvers' iteration numbers for various values of $h$ and $\varepsilon$ at $T= 0.02$, with time step $\dt= 10^{-3}$. The other common parameters are set as $A=\nicefrac{1}{16}$, $\Omega=[0, 3.2]^2$. The figures in the top row of Figure~\ref{fig:complexitybdf2etf} indicate that the convergence rate (as gleaned from the error reduction) is nearly uniform and nearly independent of $h$ for a fixed $\varepsilon$.  And the plots in the bottom row of Figure~~\ref{fig:complexitybdf2etf} show that the number of preconditioned solvers' iterations increases with a decreasing value of $\varepsilon$, which confirms the theoretical results that the PSD solver is dependent on parameter $\varepsilon$ in~\cite{feng2016preconditioned}. Figure~\ref{fig:complexitybdf2etf} confirms the expected geometric convergence rate of the PSD solver predicted by the theory in~\cite{feng2016preconditioned}. Moreover, the number of the interation steps in Figure~\ref{fig:complexitybdf2etf} also indicate that PNCG2 is the most efficient one and PNCG1 is better than PSD, especially when $\varepsilon$ is small.

\begin{figure}[!htp]
	\begin{center}
		\begin{subfigure}{0.32\textwidth}
			\includegraphics[width=\textwidth]{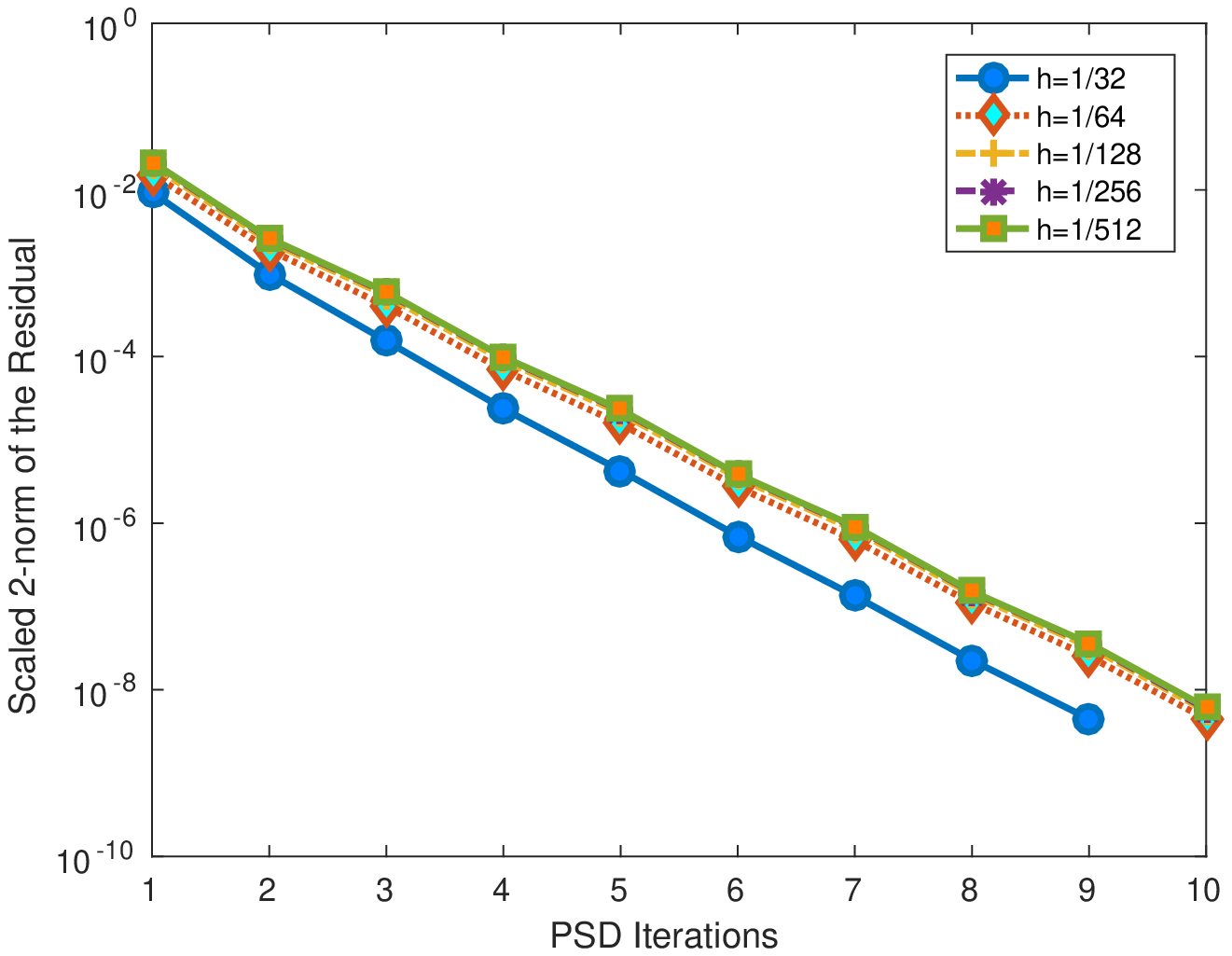}
		\end{subfigure}
		\begin{subfigure}{0.32\textwidth}
			\includegraphics[width=\textwidth]{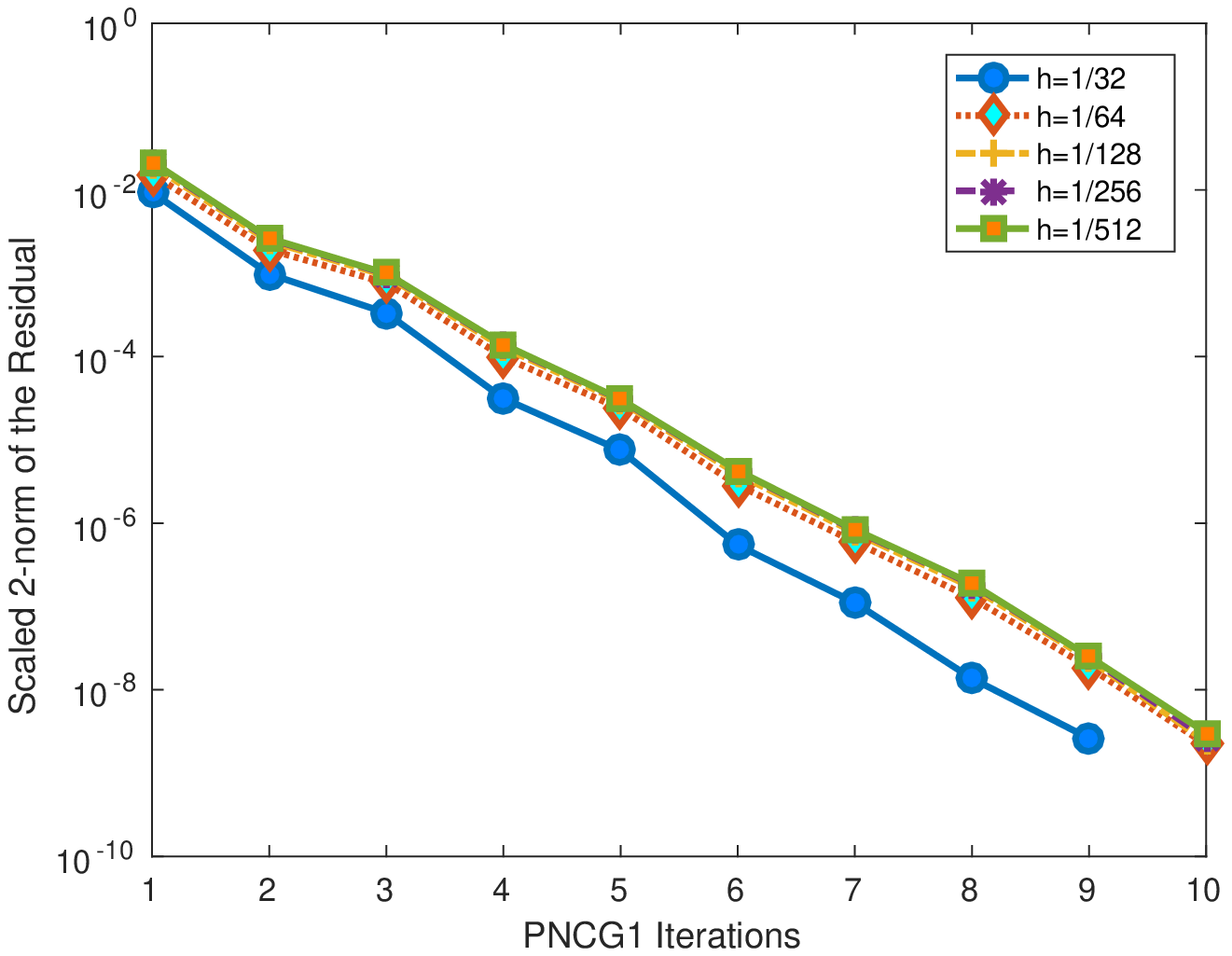} 
		\end{subfigure}
		\begin{subfigure}{0.32\textwidth}
			\includegraphics[width=\textwidth]{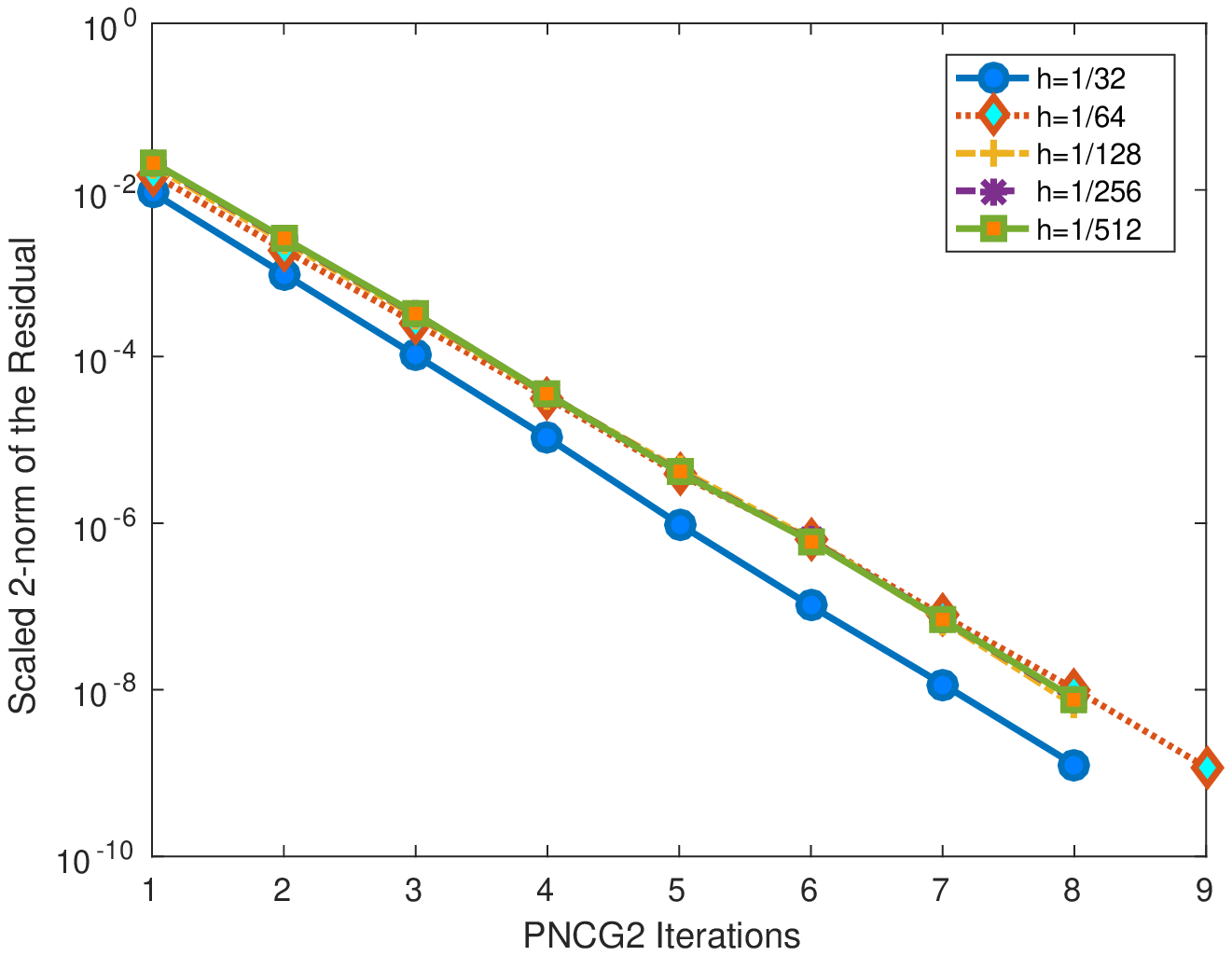} 
		\end{subfigure}	
		\begin{subfigure}{0.32\textwidth}
			\includegraphics[width=\textwidth]{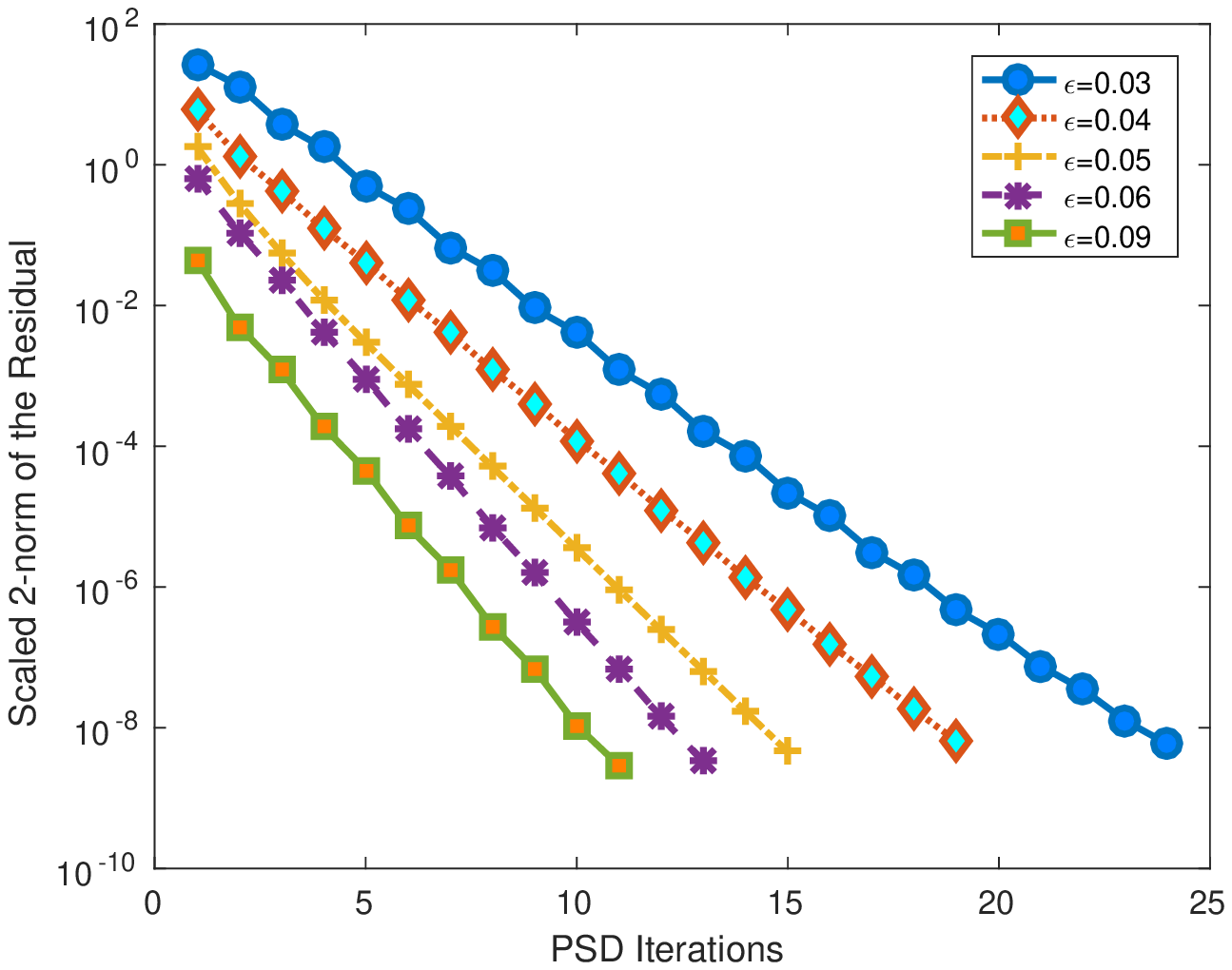}
		\end{subfigure}
		\begin{subfigure}{0.32\textwidth}
			\includegraphics[width=\textwidth]{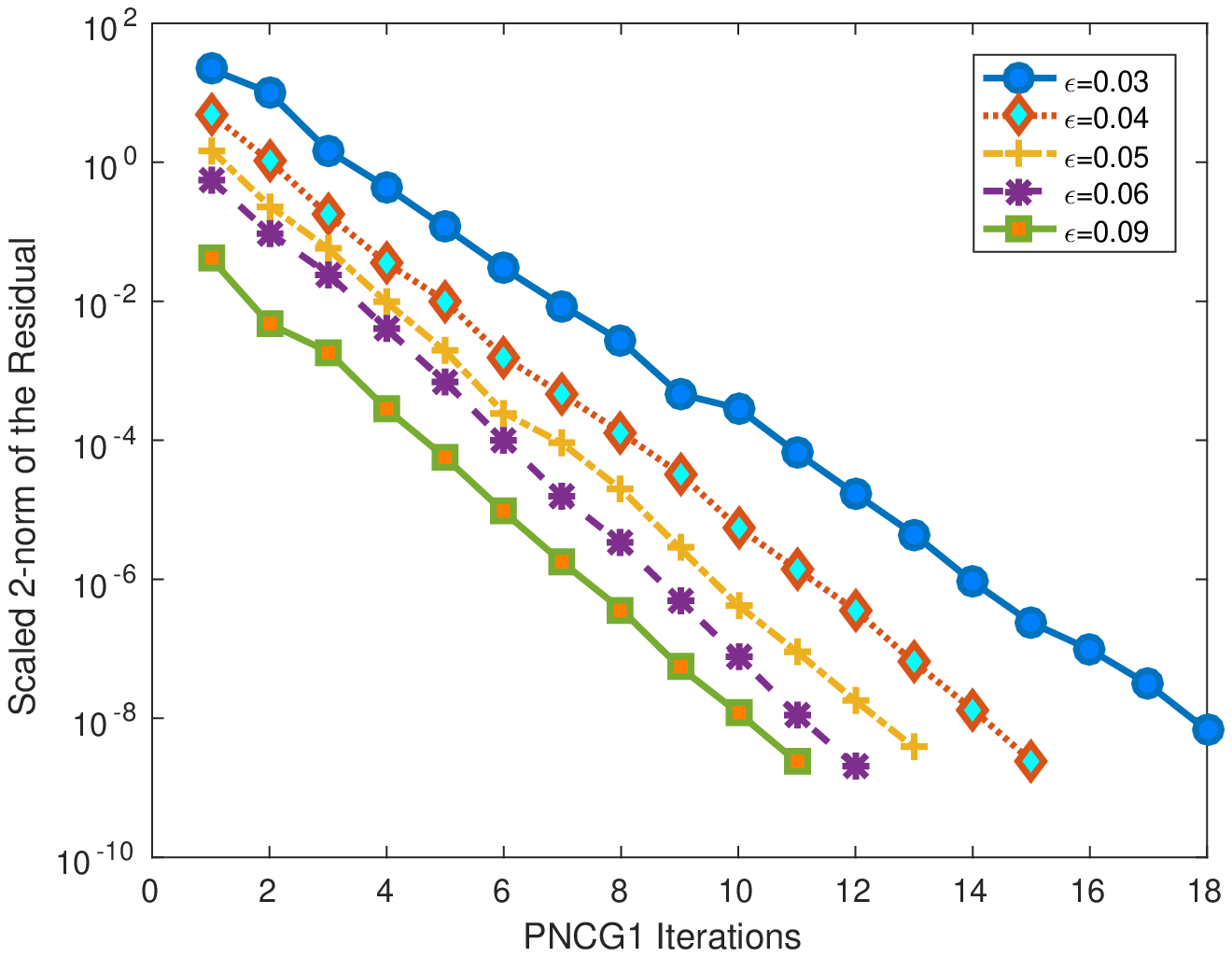} 
		\end{subfigure}
		\begin{subfigure}{0.32\textwidth}
			\includegraphics[width=\textwidth]{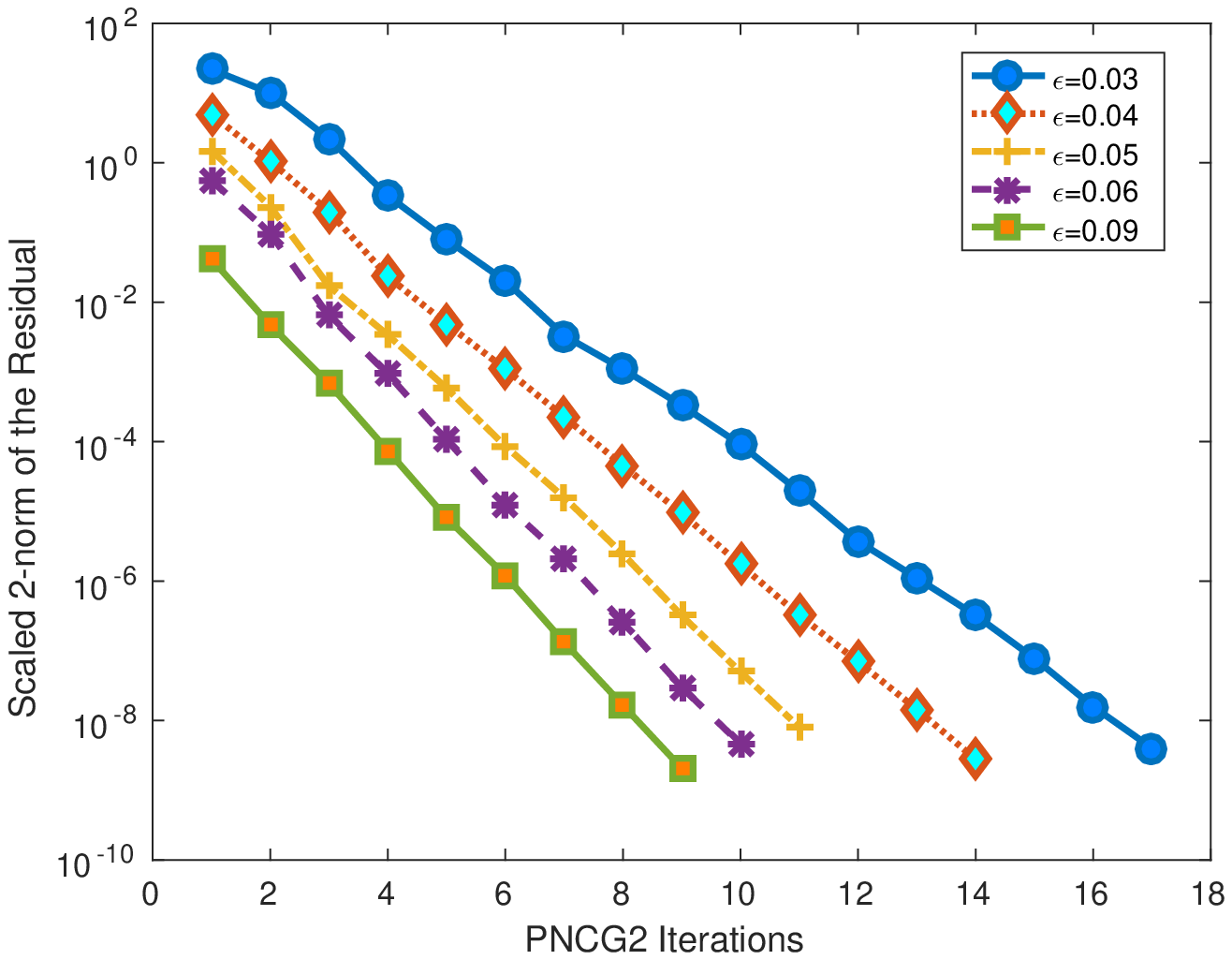} 
		\end{subfigure}				
	\end{center}
	\caption{ Complexity tests showing the solvers' performance for changing values of $h$ and $\varepsilon$. Top row: $h$-independence with $\varepsilon=0.1$; Bottom row: $\varepsilon$-dependence with $h=\nicefrac{3.2}{512}$. The rest of the parameters are given in the text. }
	\label{fig:complexitybdf2etf}
\end{figure}

In the third part of this test, we perform CPU time comparison between the proposed preconditioned solvers and the PSD solver with random initial data.
The initial data for the simulations are taken as essentially random:
\begin{equation}\label{eqn:init_random}
u^0_{i,j}=0.05\cdot(2r_{i,j}-1),
\end{equation}
where the $r_{i,j}$ are uniformly distributed random numbers in [0, 1]. The parameters for the comparison simulations are $\Omega=[0, 12.8]^2$, $\varepsilon = 3\times 10^{-2} $, $h=\nicefrac{12.8}{512}$, $\dt=0.001$ and $T= 1$. The average iteration numbers, total CPU time (in seconds) and speedups for the preconditioned methods can be found in Table~\ref{tab:cpu-fixed}.  The Table~\ref{tab:cpu-fixed} indicates that the PNCG1 solver and PNCG2 solver have provided a 1.37x and 1.45x speedup over PSD solver, respectively. 

\begin{table}[!htb]
\begin{center}
\caption{The average iteration numbers and total CPU time (in seconds) for the preconditioned methods with fixed time steps $\dt = 0.001$. Parameters are given in the text.} \label{tab:cpu-fixed}
\begin{tabular}{|c|c|c|c|}
\hline Methods & PSD & PNCG1 & PNCG2 \\
\hline  $\#_{iter}$ & 20 & 14 & 13\\
\hline  $T_{cpu}$(s) & 4406.1764 & 3212.2898  & 3035.4369 \\
\hline  Speedup & - & 1.37  & 1.45\\
\hline
\end{tabular}
\end{center}
\end{table}

In the fourth part of this test, we investigate the effect of the parameters $\dt$ and $A$ for the energy $F_h(\phi)$ with initial data \eqref{eqn:init2nd}. Since the proposed solvers give the same results, we only present the results from PSD solver in the rest of the paper. The evolutions of the energy with various time steps $\dt$ and stabilized parameter $A$ are given in Figure~\ref{fig:complexity-bdfetf-energy}. As can be seen in Figure~\ref{fig:complexity-bdfetf-energy}(a),
the larger time steps produce inaccurate or nonphysical solutions. In turn, Figure~\ref{fig:complexity-bdfetf-energy}(a) indicates the proper time steps and provides the motivation of using adaptive time stepping strategy. Figure~\ref{fig:complexity-bdfetf-energy}(b) shows that the proposed scheme and PSD solver is not that sensitive to the stabilized parameter $A$ when $A\leq1$. 

\begin{figure}[!htp]
	\begin{center}
		\begin{subfigure}{0.45\textwidth}
			\includegraphics[width=\textwidth]{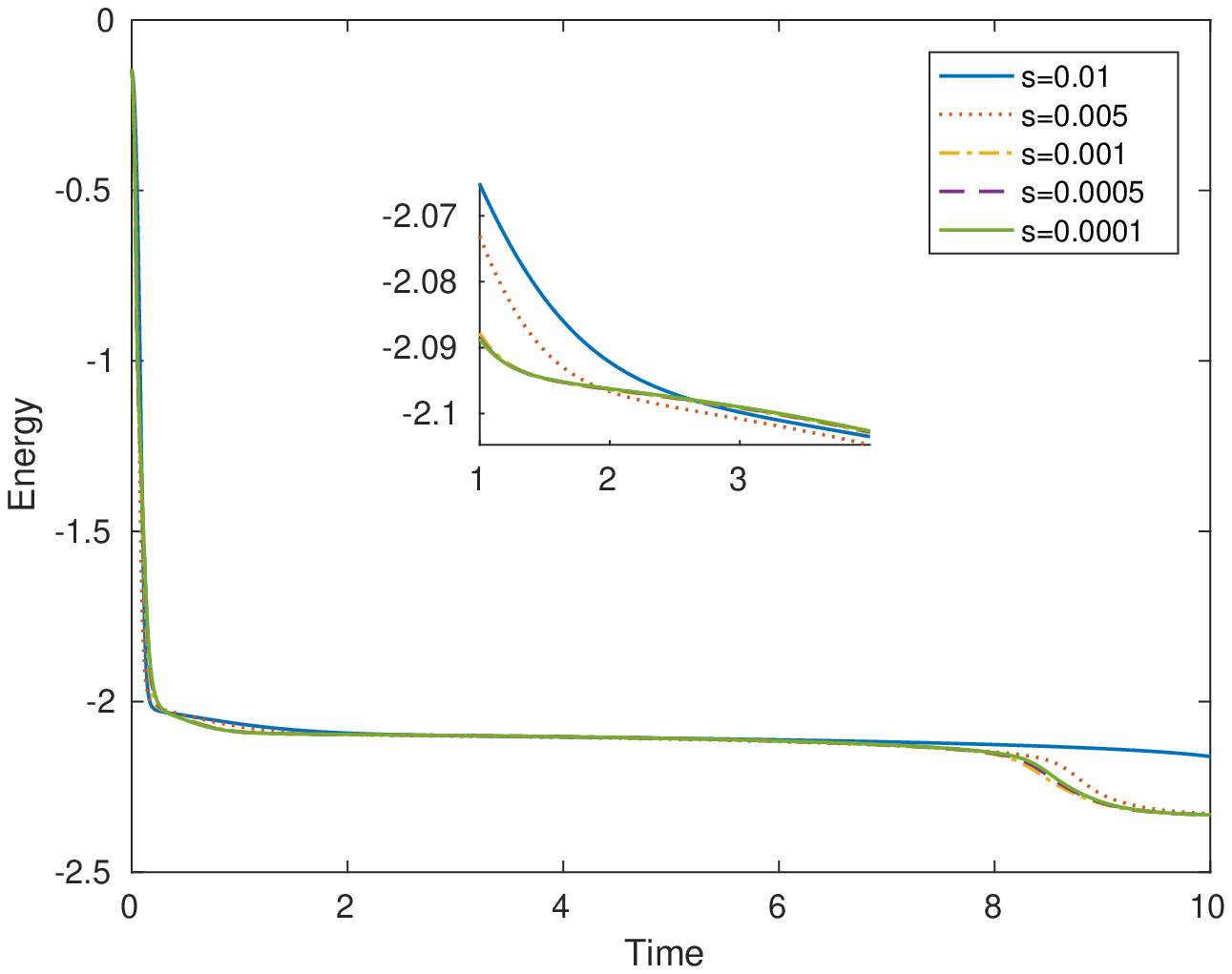}
			\caption{evolutions of energy w.r.t various $\dt$}
		\end{subfigure}
		\begin{subfigure}{0.45\textwidth}
			\includegraphics[width=\textwidth]{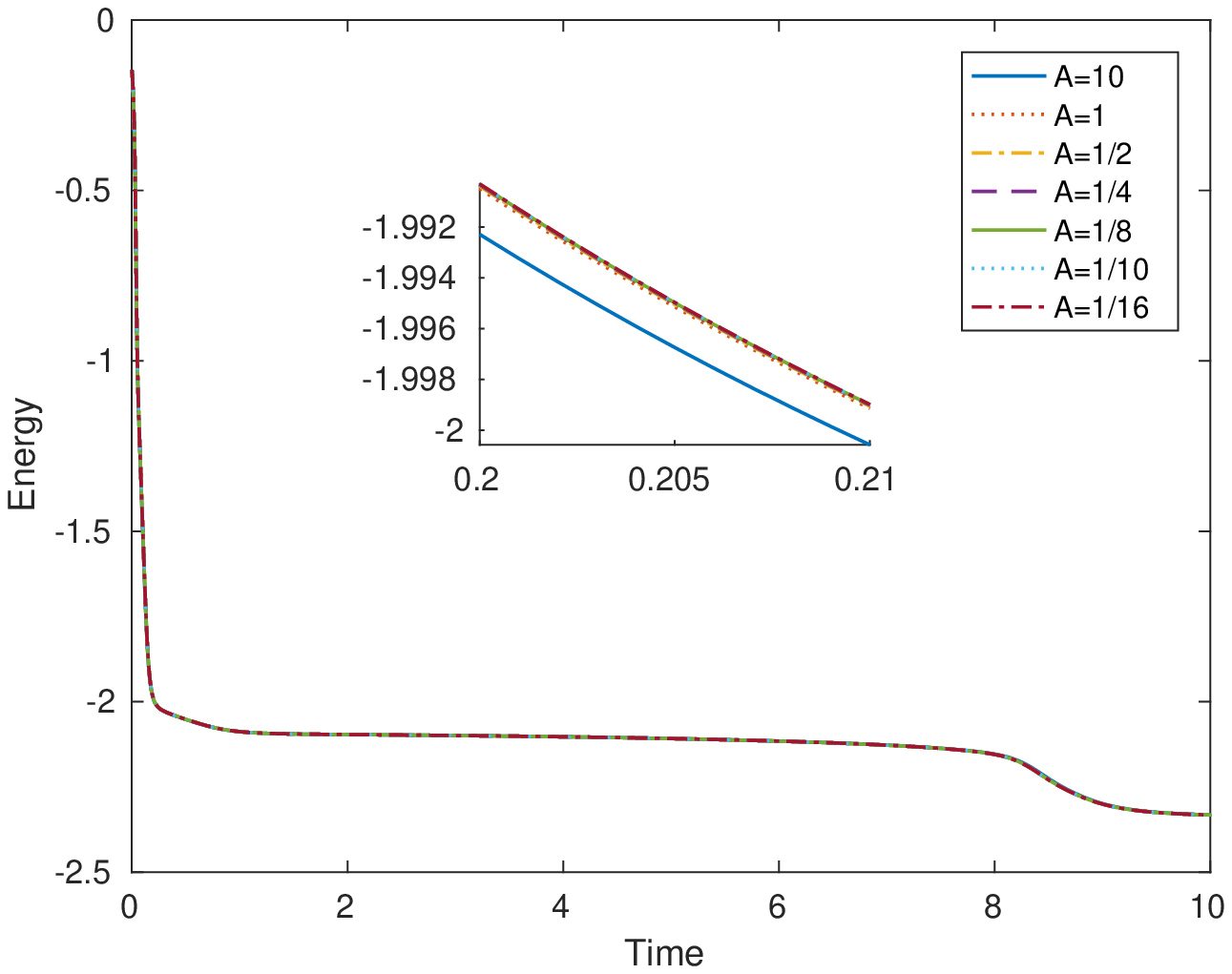} 
			\caption{evolutions of energy w.r.t various $A$}
		\end{subfigure}
	\end{center}
	\caption{ The effect of time steps $\dt$ and stabilized parameter $A$  for the energy $F_h(\phi)$. 
Left: the effect of  time step $\dt$. The other parameters are  $\Omega=[0, 3.2]^2$, $\epsilon = 3.0\times 10^{-2}$, $h=\nicefrac{3.2}{512}$, and $A=\nicefrac{1}{16}$; Right: the effect of stabilized parameter $A$. The other parameters are  $\Omega=[0, 3.2]^2$, $\epsilon = 3.0\times 10^{-2}$, $h=\nicefrac{3.2}{512}$ and $\dt = 0.001 $.  
	}
	\label{fig:complexity-bdfetf-energy}
\end{figure}

\subsection{Long-time coarsening process, energy dissipation and mass conservation }
Coarsening processes in thin film system can take place on very long time scales~\cite{kohn06}.  In this subsection, we perform  long time simulation for the SS equation.   Such a test, which has been performed in many existing literature, will  confirm the expected coarsening rates and serve as a benchmarks for the proposed solver; see, for example, \cite{feng2016preconditioned, shen2012second, wang10a}.

The initial data for this simulations are taken as \eqref{eqn:init_random}. Time snapshots of the evolution for the epitaxial thin film growth model can be found in Figure \ref{fig:long-time-psd-2nd}. The coarsening rates  are given in Figure~\ref{fig:one-third-psd-2nd}.  The interface width or roughness is defined as
 \begin{eqnarray}
W(t_n)=\sqrt{\frac{h^2}{mn}\sum_{i=1}^{m}\sum_{j=1}^{n}(\phi_{i,j}^n-\bar{\phi})^2},
\end{eqnarray}
where $m$ and $n$ are the number of the grid points in $x$ and $y$ direction and $\bar{\phi}$ is the average value of $\phi$ on the uniform grid.  The log-log plots of roughness and energy evolution and the corresponding linear regression are presented in Figure.~\ref{fig:one-third-psd-2nd}. The linear regression in Figure.~\ref{fig:one-third-psd-2nd} indicates that the surface roughness grows like $t^{1/3}$, while the energy decays like $t^{-1/3}$, which verifies the one-third power law predicted in \cite{kohn03}. More precisely, the linear fits have the form $a_et^{b_e}$ with $a_e = 3.09870, b_e = -0.33554$ for energy evolution and $a_mt^{b_m}$ with $a_m = -5.35913, b_m = 0.32555$ for roughness evolution. The linear regression is only taken up to $t = 3000$, since the saturation time would be of the order of $\varepsilon^{-2}$ under the scaling that we have adopted \cite{shen2012second}. These simulation results are consistent with earlier works on this topic in \cite{feng2016preconditioned,shen2012second, wang10a,xu06}. 
\begin{figure}[!htp]
	\begin{center}
		\begin{subfigure}{0.45\textwidth}
			\includegraphics[width=\textwidth]{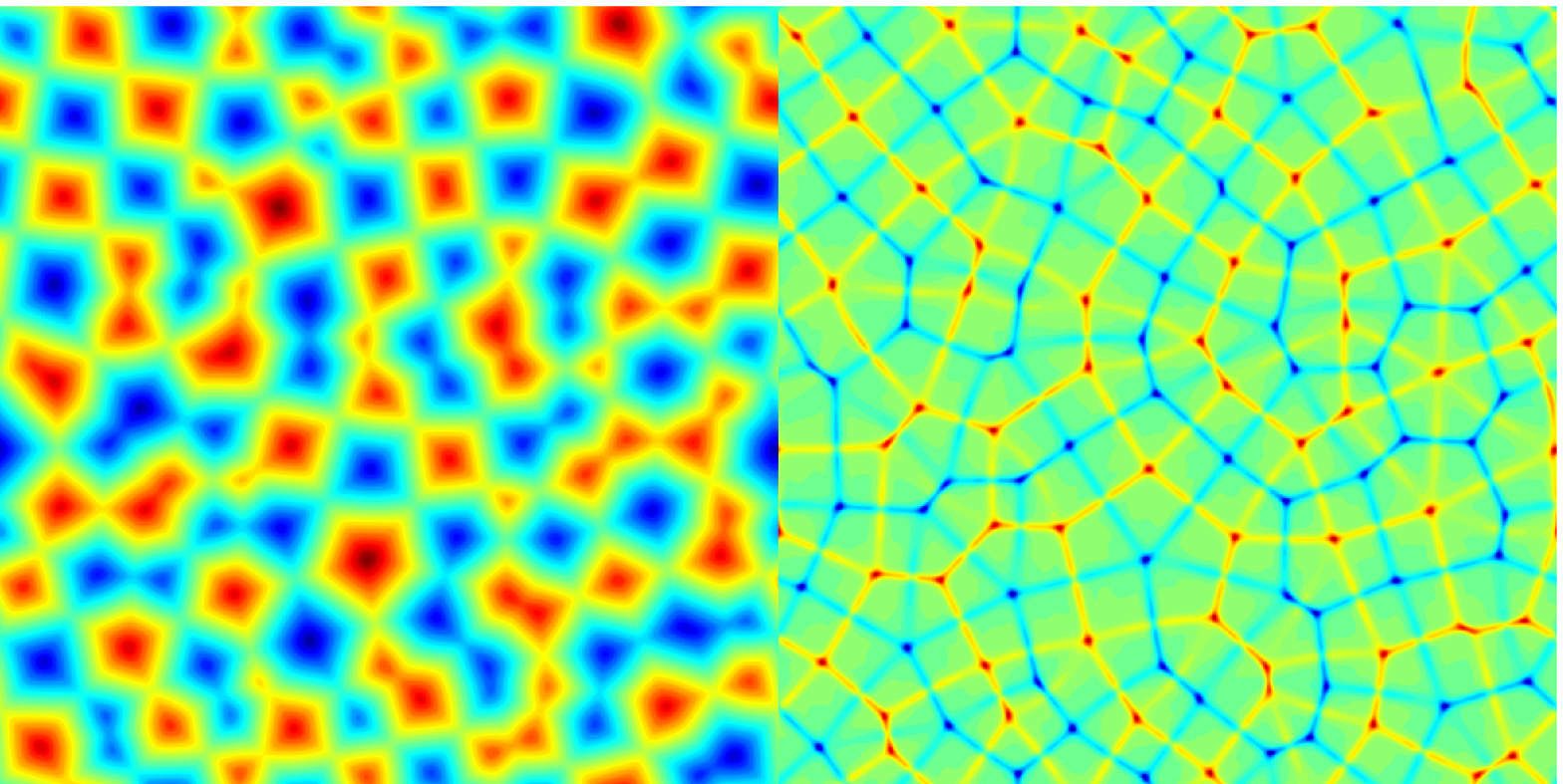} 
			\caption*{$t=10$}
		\end{subfigure}
		\begin{subfigure}{0.45\textwidth}
			\includegraphics[width=\textwidth]{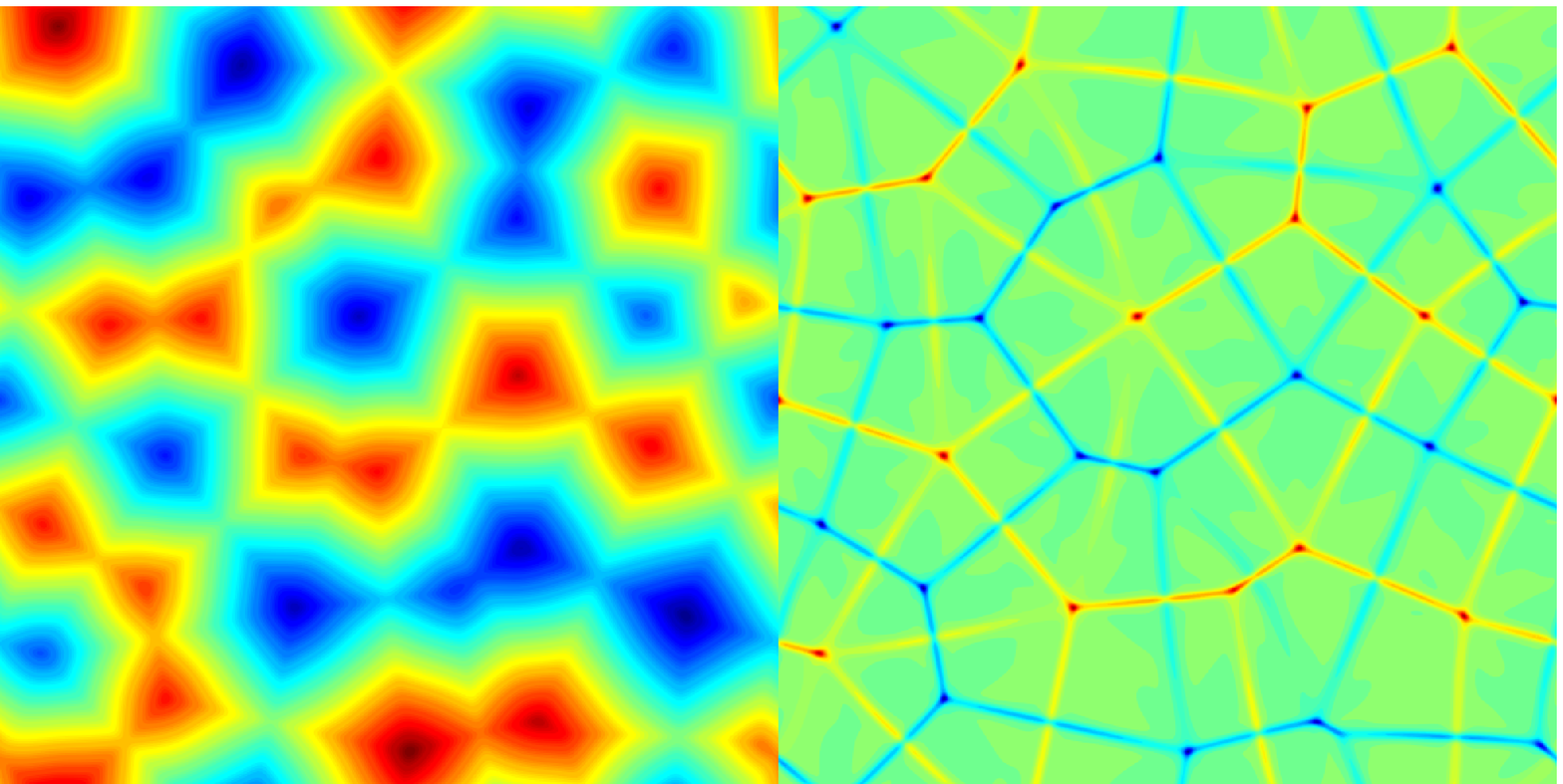}
			\caption*{$t=100$}
		\end{subfigure}
		\begin{subfigure}{0.45\textwidth}
			\includegraphics[width=\textwidth]{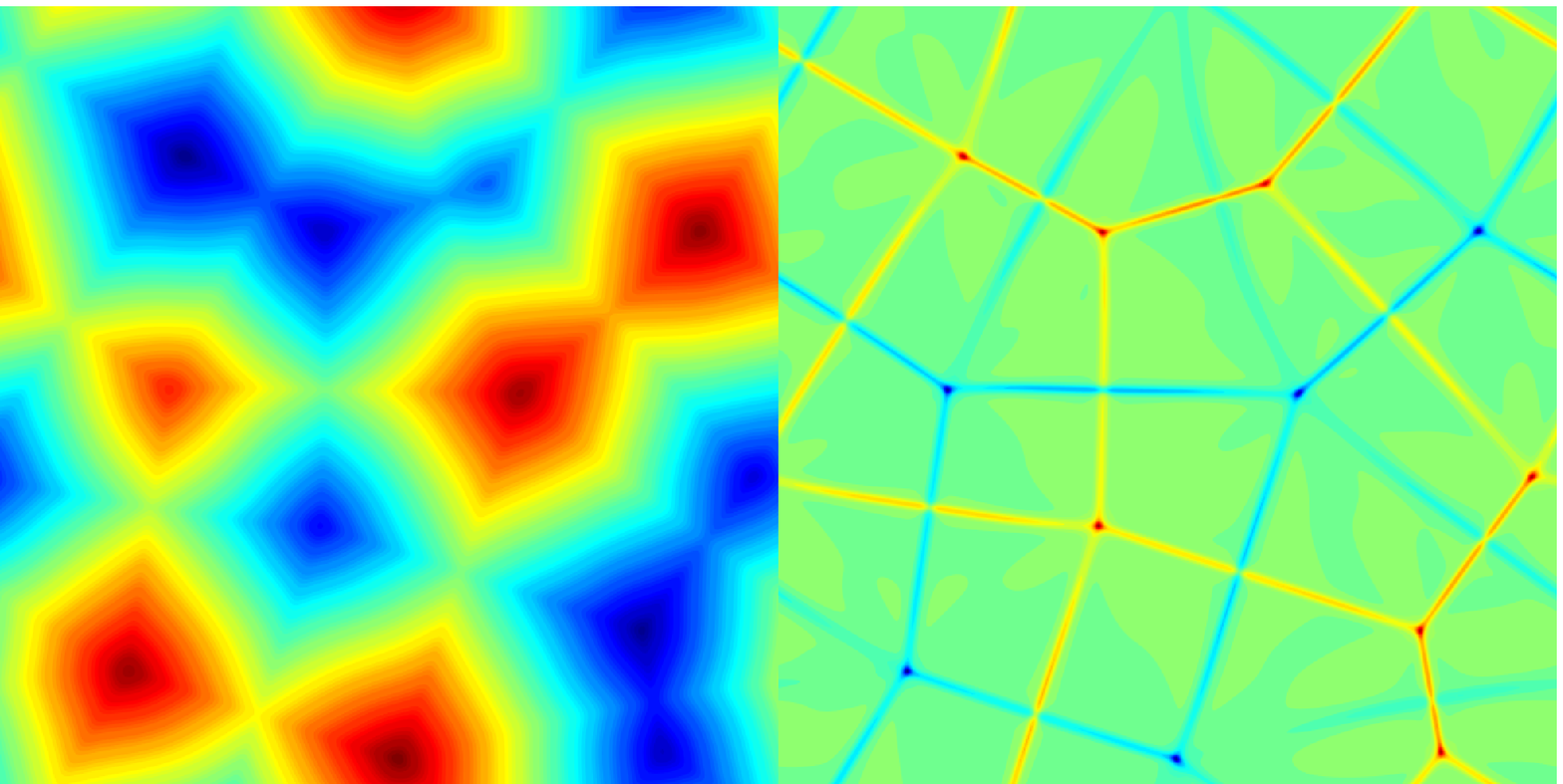} 
			\caption*{$t=500$}
		\end{subfigure}
		\begin{subfigure}{0.45\textwidth}
			\includegraphics[width=\textwidth]{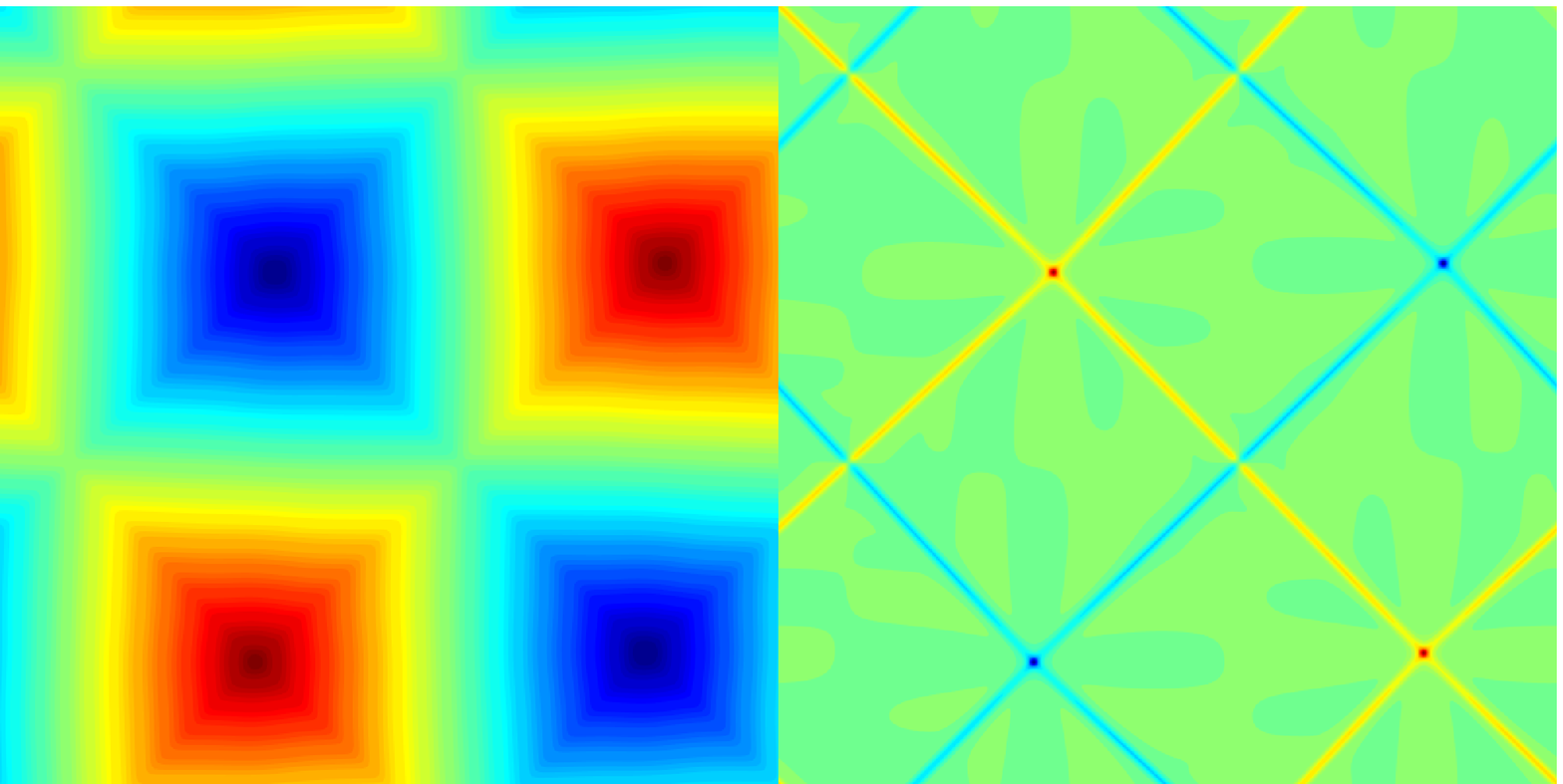}
			\caption*{$t=2000$}
		\end{subfigure}
		\begin{subfigure}{0.45\textwidth}
			\includegraphics[width=\textwidth]{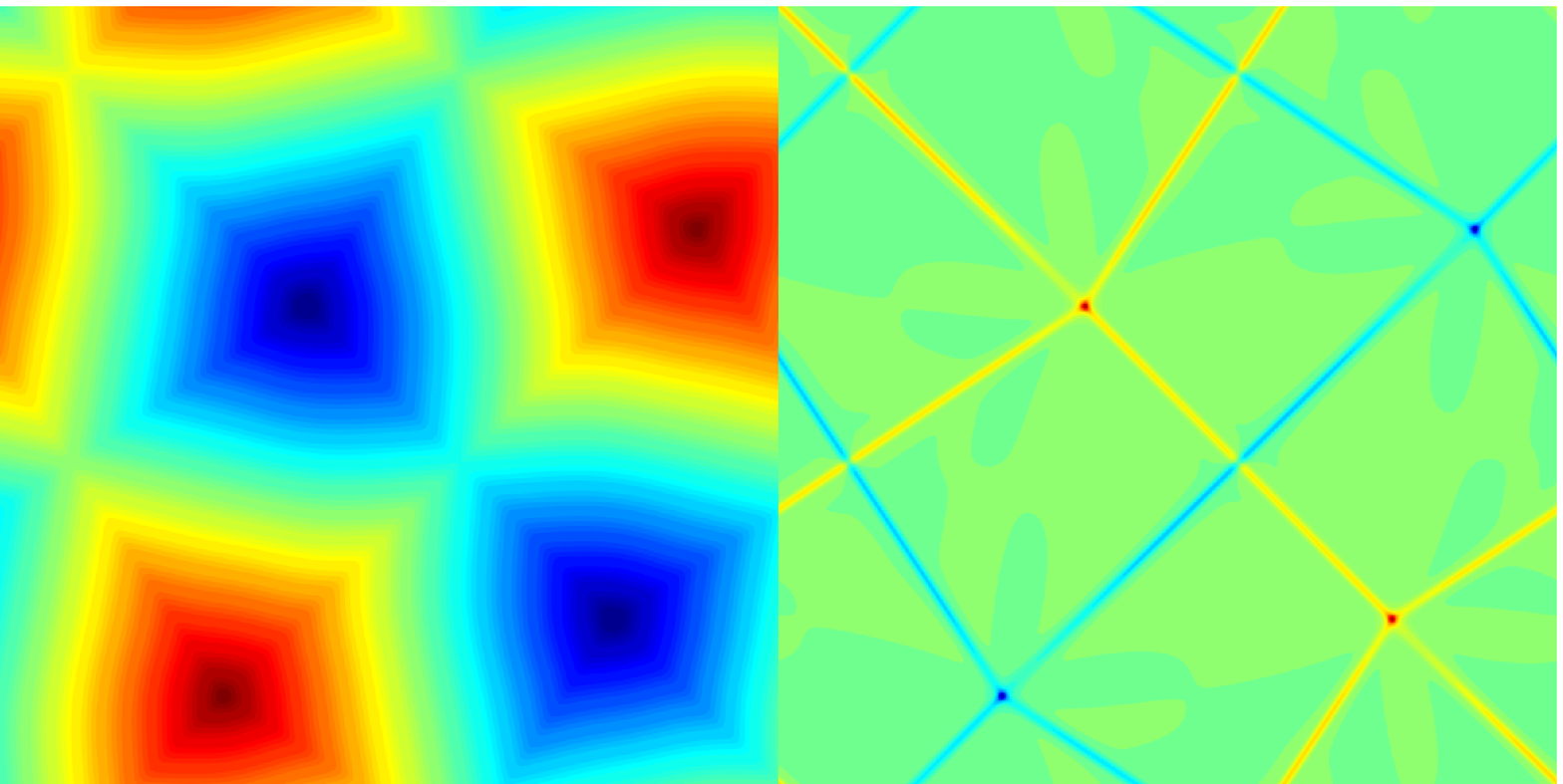} 
			\caption*{$t=4000$}
		\end{subfigure}
		\begin{subfigure}{0.45\textwidth}
			\includegraphics[width=\textwidth]{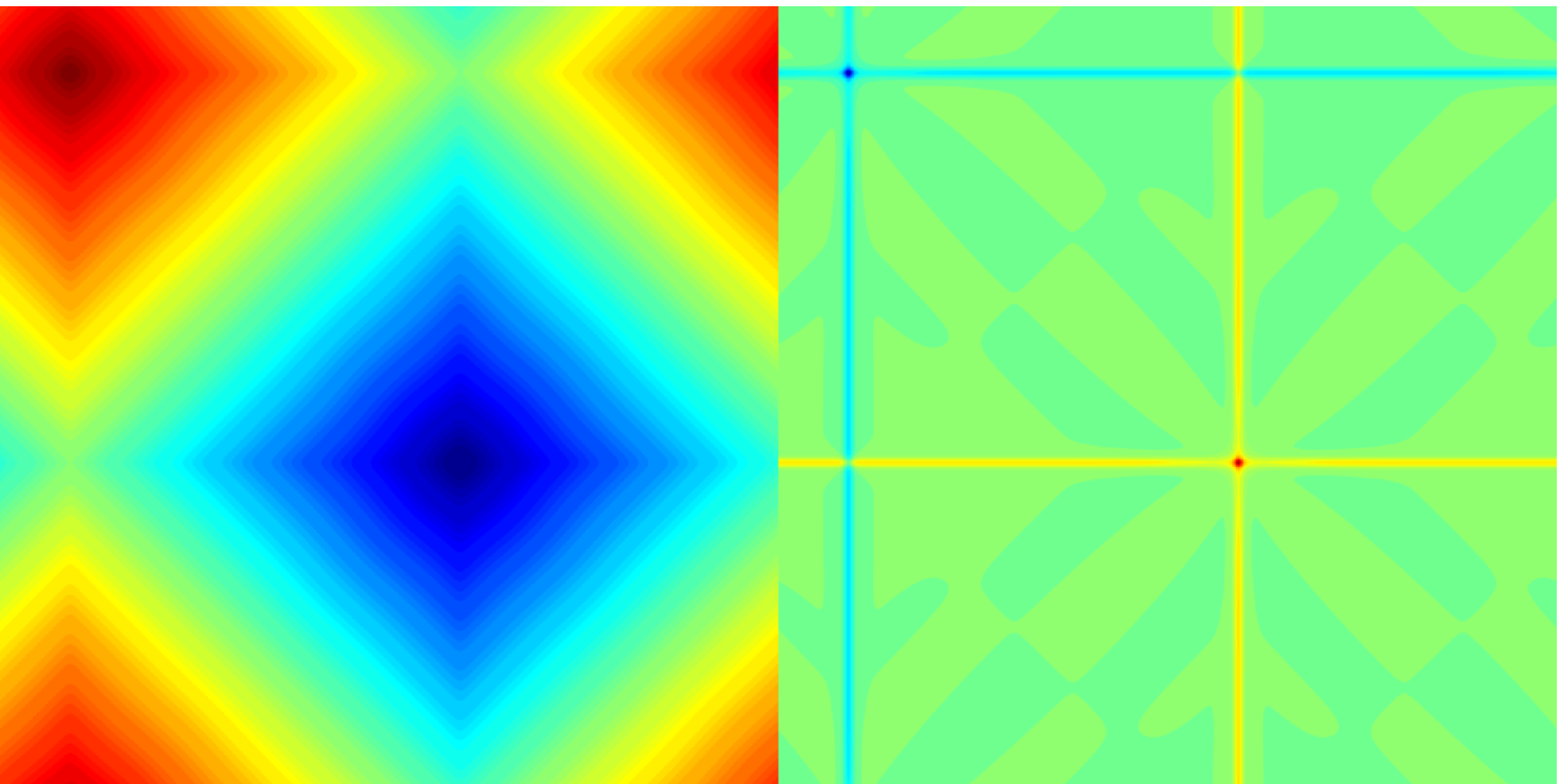}
			\caption*{$t=10000$}
		\end{subfigure}
		\caption{Time snapshots of the evolution with preconditioned solvers for the epitaxial thin film growth model at $t=10,100, 500, 2000, 4000~ \text{and}~ 10000$. Left: contour plot of $ u $, Right: contour plot of $\Delta u $. The parameters are
			$\varepsilon = 0.03, \Omega=[12.8]^2, s=0.001$, $h=\nicefrac{12.8}{512}$ and $A=\nicefrac{1}{16}$.
			These simulation results are consistent with earlier work on this topic in \cite{feng2016preconditioned,shen2012second, wang10a,xu06}.}
		\label{fig:long-time-psd-2nd}
	\end{center}
\end{figure}

\begin{figure}[!htp]
	\begin{center}
		\begin{subfigure}{0.45\textwidth}
			\includegraphics[width=\textwidth]{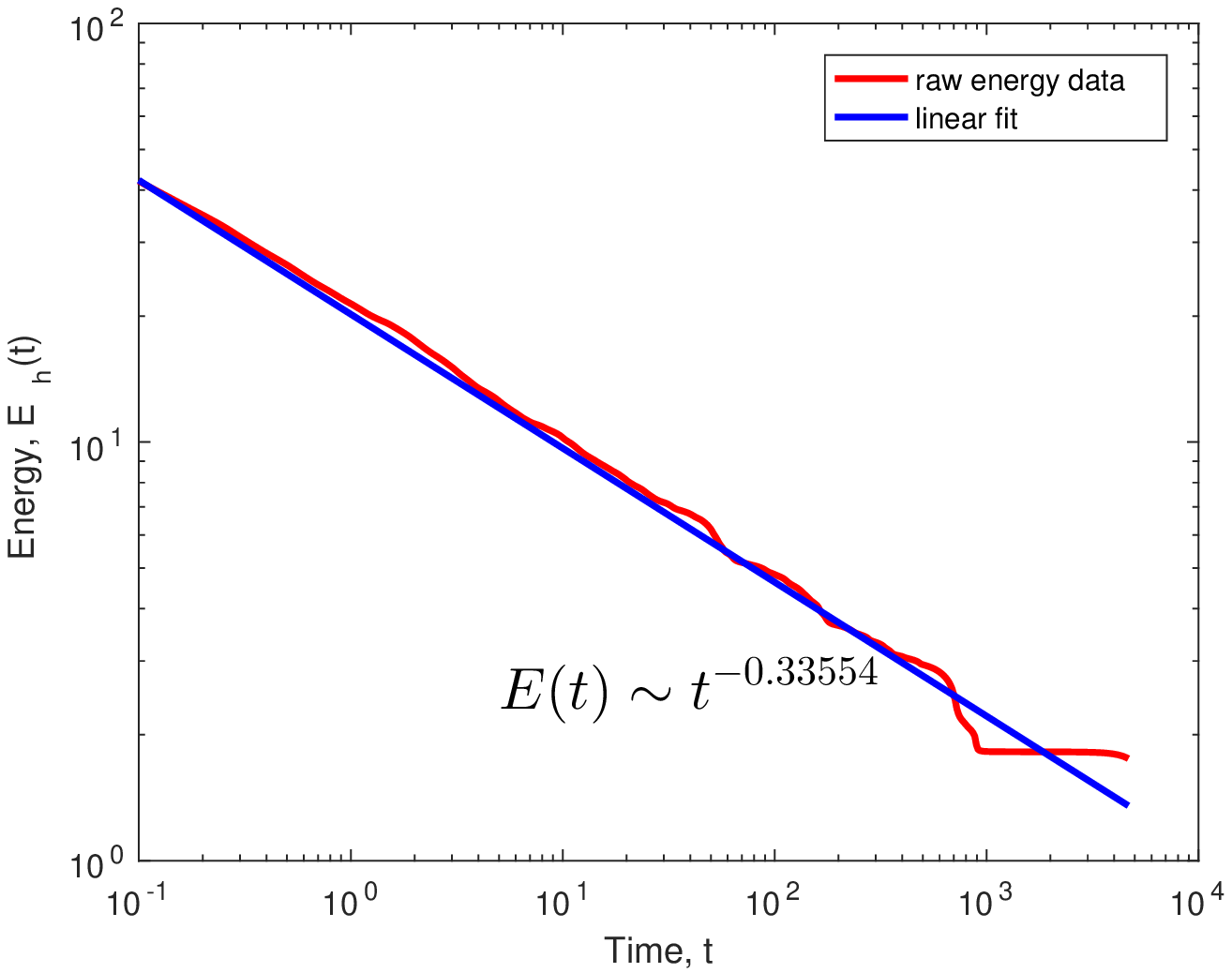}
			\caption{Energy evolution}
		\end{subfigure}
		\begin{subfigure}{0.45\textwidth}
			\includegraphics[width=\textwidth]{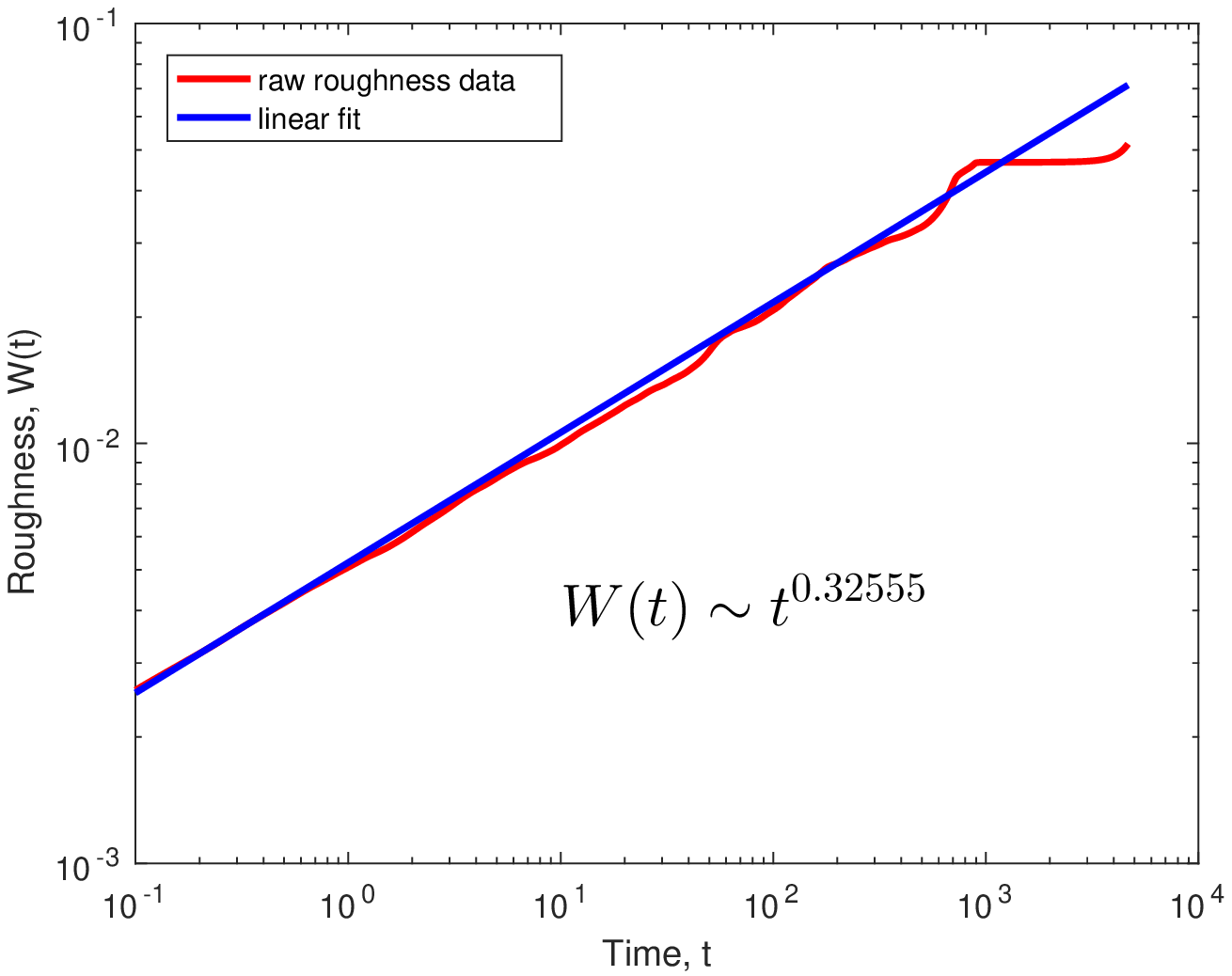} 
			\caption{Roughness evolution}
		\end{subfigure}
	\end{center}
	\caption{The log-log plots of energy and roughness evolution and the corresponding linear regression for the simulation depicted in Figure~\ref{fig:long-time-psd-2nd}. }
	\label{fig:one-third-psd-2nd}
\end{figure}

\section{Conclusions}  \label{sec:conclusion} 
In this paper, we have proposed and analyzed a second order accurate, unconditionally energy stable finite difference scheme for solving the two-dimensional epitaxial thin film  with Slope Selection (SS) equation. The unique solvability, unconditional energy stability and optimal convergence analysis have been theoretically justified. In addition, a class of efficient preconditioned methods are applied to solve the nonlinear system. This framework can be easily generalized to the higher order in time BDF schemes. Various numerical results are also presented, including the second-order-in-time accuracy test, complexity test and one-third law test. 

	\section{Acknowledgments}
 The first author would like to thank the Beijing Normal University  for support during his visit. This work is supported in part by NSF DMS-1418689 (C.~Wang), NSF DMS-1418692 (S.~Wise),  NSFC 11271048, 91130021 and the Fundamental Research Funds for the Central Universities (Z.~Zhang). 
 
 	\appendix
	
	\section{Proof of Proposition~\ref{prop:1}}
	\label{appen:A}
	
For simplicity of presentation, in the analysis of $\| \nabla_h \phi \|_6$, we are focused on the estimate of $D_x \phi \|_6$. 
Due to the periodic boundary conditions for $\phi$ and its cell-centered representation, it has a corresponding discrete Fourier transformation:
	\begin{eqnarray}
\phi_{i,j} &=& \sum^{K}_{\ell,m=-K}
\hat{\phi}^N_{\ell,m} {\rm e}^{2 \pi i ( \ell x_{i} + m y_{j} )/ L }  ,
   \label{def:Fourier-1}
	\end{eqnarray}
where $x_{i} = (i - \frac12 ) h$, $y_{j} = ( j - \frac12) h$, and $\hat{\phi}^N_{\ell,m}$ 
are discrete Fourier coefficients. Then we make its extension to a continuous function:
	\begin{equation}
	\label{def:extension-1}
\phi_{{\bf F}}(x,y) = \sum^{K}_{\ell,m=-K} \hat{\phi}^N_{\ell,m} {\rm e}^{2 \pi i ( \ell x + m y)/ L }  .
	\end{equation}
	
Similarly, we denote a grid function $f_{i+\hf,j+\hf} = \mD_x \phi_{i+\hf,j+\hf} = A_y(D_x\phi)_{i+\hf,j+\hf}$. The periodic boundary conditions for $f$ and its mesh location indicates the following discrete Fourier transformation:
	\begin{eqnarray}
f_{i+\hf,j+\hf} &=& \sum^{K}_{\ell,m=-K}
\hat{f}^N_{\ell,m} {\rm e}^{2 \pi i ( \ell x_{i+\hf} + m y_{j+\hf} )/ L } , 
   \label{def:Fourier-2}
	\end{eqnarray}
with $\hat{f}^N_{\ell,m}$ the discrete Fourier coefficients. And also, its extension to a continuous function is given by 
	\begin{equation}
	\label{def:extension-2}
f_{{\bf F}}(x,y) = \sum^{K}_{\ell,m=-K} \hat{f}^N_{\ell,m} {\rm e}^{2 \pi i ( \ell x + m y)/ L }  .
	\end{equation}
	
Meanwhile, we also observe that $\hat{\phi}^N_{0,0}=0$ and $\hat{f}^N_{0,0}=0$. The first identity comes from the fact that $\overline{\phi}=0$, while the second one is due to the fact that $\overline{f} = \overline{\mD_x \phi} =0$, for any periodic grid function $\phi$. 

The following preliminary estimates will play a very important role in the later analysis. 

\begin{lem} \label{lem:A.1}
We have 
\begin{eqnarray} 
  &&
  \| \phi \|_2 = \| \phi_{\bf F} \| ,  \label{lemma A.1-1} 
\\
  && 
  \frac{2}{\pi} \| \nabla \phi_{\bf F} \| \le \nrm{\nabla_h \phi}_2 \le \| \nabla \phi_{\bf F} \|  ,  \quad 
  \frac{4}{\pi^2} \| \Delta \phi_{\bf F} \| \le \nrm{\Delta_h \phi}_2 \le \| \Delta \phi_{\bf F} \| ,  
  \label{lemma A.1-2} 
\\
  && 
  \nrm{\partial_x f_{\bf F} } \le \nrm{\partial_x^2 \phi_{\bf F} } ,  \quad 
  \nrm{\partial_y f_{\bf F} } \le \nrm{\partial_x \partial_y \phi_{\bf F} } .  \label{lemma A.1-3} 
\end{eqnarray}    
\end{lem} 

	\begin{proof}
Parseval's identity (at both the discrete and continuous levels) implies that
\begin{eqnarray}
   \sum^{N-1}_{i,j=0}|\phi_{i,j}|^2 =  N^2 \sum^{K}_{\ell,m=-K} 
|\hat{\phi}^N_{\ell,m,n}|^2,  \quad 
  \nrm{\phi_{\bf F}}^2 = L^2  \sum^{K}_{\ell,m=-K}
|\hat{\phi}^N_{\ell,m}|^2. \label{lemma A.1-0-2} 
\end{eqnarray}
Based on the fact that $h N = L$, this in turn results in
\begin{equation}
\nrm{\phi}^2_{2} = \nrm{\phi_{{\bf F}}}^2 =  L^2 \sum^{K}_{\ell,m=-K}|\hat{\phi}^N_{\ell,m}|^2 ,  \label{lemma A.1-5}
\end{equation}
so that (\ref{lemma A.1-1}) is proven. 

For the comparison between $f = \mD_x \phi$ and $\partial_x \phi_{\bf F}$, we look at the following Fourier expansions:
	\begin{eqnarray}
f_{i+\hf,j+\hf} &=& \frac{\phi_{i+1,j}-\phi_{i,j} + \phi_{i+1,j+1} - \phi_{i,j+1}}{2h} 
 = \sum^{K}_{\ell,m=-K} \mu_{\ell,m} \hat{\phi}^N_{\ell,m}  {\rm e}^{2 \pi i  ( \ell x_{i+\hf} + m y_{j+\hf} )/ L }  ,  \label{lemma A.1-6}
	\\
f_{\bf F} (x,y) &=& \sum^{K}_{\ell,m=-K} \mu_{\ell,m} \hat{\phi}^N_{\ell,m}  {\rm e}^{2 \pi i  ( \ell x + m y )/ L }  ,  \label{lemma A.1-7}
\\	
  \partial_x \phi_{{\bf F}} (x,y) &=& \sum^{K}_{\ell,m=-K}  
  \nu_{\ell} \hat{\phi}^N_{\ell,m} {\rm e}^{2 \pi i ( \ell x + m y ) / L } , \label{lemma A.1-8}  
	\end{eqnarray}
with
	\begin{equation}
\mu_{\ell,m} = -\frac{2 i \sin{\frac{\ell\pi h}{L}}}{h} \cos (m \pi h) , \quad
\nu_{\ell} = -\frac{2 \ell\pi i}{L}. \label{lemma A.1-9}
\end{equation}
A comparison of Fourier eigenvalues between $|\mu_{\ell, m}|$ and $|\nu_{\ell}|$ shows that
\begin{equation}
\frac{2}{\pi} |\nu_{\ell}| \le |\mu_{\ell, m}| \le |\nu_{\ell}|, 
\quad \rm{for}  \quad -K \le \ell, m \le K , \label{lemma A.1-10}
\end{equation}
which in turn leads to 
\begin{eqnarray} 
  \frac{2}{\pi} \| \partial_x \phi_{\bf F} \| \le \nrm{\mD_x \phi}_2 \le \| \partial_x \phi_{\bf F} \|  .   
  \label{lemma A.1-10-2} 
\end{eqnarray} 
A similar estimate could also be derived: 
\begin{eqnarray} 
  \frac{2}{\pi} \| \partial_y \phi_{\bf F} \| \le \nrm{\mD_y \phi}_2 \le \| \partial_y \phi_{\bf F} \|  .   
  \label{lemma A.1-10-3} 
\end{eqnarray}
A combination of \eqref{lemma A.1-10-2} and \eqref{lemma A.1-10-3} yields the first inequality of (\ref{lemma A.1-2}).  

For the second estimate of (\ref{lemma A.1-2}), we look at similar Fourier expansions: 
	\begin{eqnarray}
(\Delta_h \phi)_{i,j} &=& \sum^{K}_{\ell,m=-K} \left( \mu_{\ell}^2 + \mu_{m}^2 \right) \hat{\phi}^N_{\ell,m}  {\rm e}^{2 \pi i  ( \ell x_{i} + m y_{j} )/ L }  ,  \label{lemma A.1-11}
	\\	
  \Delta \phi_{{\bf F}} (x,y) &=& \sum^{K}_{\ell,m=-K}  
  \left( \nu_{\ell}^2 + \nu_m^2 \right) \hat{\phi}^N_{\ell,m} {\rm e}^{2 \pi i ( \ell x + m y ) / L } ,  \label{lemma A.1-12}  
	\end{eqnarray}
with $\mu_{\ell} = -\frac{2 i \sin{\frac{\ell\pi h}{L}}}{h}$, $\mu_{m} = -\frac{2 i \sin{\frac{m \pi h}{L}}}{h}$. It is also clear that $\frac{2}{\pi} |\nu_{\ell}| \le |\mu_{\ell}| \le |\nu_{\ell}|$, for any $-K \le \ell \le K$. In turn, an application of Parseval's identity yields
\begin{eqnarray}
\nrm{\Delta_h \phi}^2_2 = L^2 \sum^{K}_{\ell,m=-K}\left| \mu_{\ell}^2 + \mu_m^2 \right|^2 |\hat{\phi}^N_{\ell,m}|^2,  \label{lemma A.1-13-1}  \\
\nrm{\Delta \phi_{\bf F}}^2 =
 L^2 \sum^{K}_{\ell,m=-K} \left| \nu_{\ell}^2 + \nu_m^2 \right|^2
  |\hat{\phi}^N_{\ell,m}|^2. \label{lemma A.1-13-2} 
\end{eqnarray}
The eigenvalue comparison estimate (\ref{lemma A.1-10}) implies the following inequality: 
\begin{equation}
\frac{4}{\pi^2} \left| \nu_{\ell}^2 + \nu_m^2 \right| \le \left| \mu_{\ell}^2 + \mu_m^2 \right| \le \left| \nu_{\ell}^2 + \nu_m^2 \right|, 
\quad \rm{for}  \quad -K \le \ell, m  \le K . \label{lemma A.1-14} 
\end{equation}
As a result, inequality (\ref{lemma A.1-2}) comes from a combination of (\ref{lemma A.1-13-1}), (\ref{lemma A.1-13-2}) and (\ref{lemma A.1-14}). 

For the estimate (\ref{lemma A.1-3}), we observe the following Fourier expansions: 
	\begin{eqnarray}
 \partial_x f_{\bf F} (x,y) &=& \sum^{K}_{\ell,m=-K} \nu_{\ell} \mu_{\ell,m} \hat{\phi}^N_{\ell,m}  {\rm e}^{2 \pi i  ( \ell x + m y )/ L }  ,  \label{lemma A.1-15}
	\\
\partial_x^2 \phi_{\bf F} (x,y) &=& \sum^{K}_{\ell,m=-K}  
  \nu_{\ell}^2 \hat{\phi}^N_{\ell,m} {\rm e}^{2 \pi i ( \ell x + m y ) / L } , \label{lemma A.1-16}  
	\end{eqnarray}
which in turn leads to (with an application of Parseval's identity) 
\begin{eqnarray}
\nrm{\partial_x f_{\bf F}}^2 = L^2 \sum^{K}_{\ell,m=-K}\left| \nu_\ell \mu_{\ell,m} \right|^2 |\hat{\phi}^N_{\ell,m}|^2,  \label{lemma A.1-17-1}  \\
\nrm{\partial_x^2 \phi_{\bf F}}^2 =
 L^2 \sum^{K}_{\ell,m=-K} | \nu_{\ell} |^4 |\hat{\phi}^N_{\ell,m}|^2. \label{lemma A.1-17-2} 
\end{eqnarray}
Similarly, the following inequality could be derived, based on the eigenvalue comparison estimate (\ref{lemma A.1-10}): 
\begin{equation}
  \left| \nu_\ell \mu_{\ell,m} \right|^2 \le | \nu_\ell |^4 ,   
\quad \rm{for}  \quad -K \le \ell, m  \le K . \label{lemma A.1-18} 
\end{equation}
Consequently, a combination of (\ref{lemma A.1-17-1}), (\ref{lemma A.1-17-2}) and (\ref{lemma A.1-18}) leads to the first inequality in (\ref{lemma A.1-3}). The second inequality, $\nrm{\partial_y f_{\bf F} } \le \nrm{\partial_x \partial_y \phi_{\bf F} }$, could be derived in the same manner. The proof of Lemma~\ref{lem:A.1} is complete.   
\end{proof} 

With the estimates in Lemma~\ref{lem:A.1}, we are able to make the following derivations: 
\begin{eqnarray} 
  &&
   \| \phi \|_{H_{h}^2}^2 = \| \phi \|_2^2 + \| \nabla_h \phi \|_2^2 + \| \Delta_h \phi \|_2^2  
   \le \| \phi_{\bf F} \|^2 + \| \nabla \phi_{\bf F} \|^2 + \| \Delta \phi_{\bf F} \|^2  
   \le \| \phi_{\bf F} \|_{H_{h}^2}^2 ,  \label{prop 1-1-1}   
\\
  &&
   \| \phi_{\bf F} \|_{H_{h}^2}^2 \le B_0 \| \Delta \phi_{\bf F} \|^2  ,  \quad 
   \mbox{(elliptic regularity, since $\int_\Omega \, \phi_{\bf F} \, d {\bf x} =0$)}, \label{prop 1-1-2}    
\\
  &&
   \mbox{so that} \, \, \, \nrm{\Delta_h \phi}_2^2  \ge \frac{4}{\pi^2} \nrm{\Delta \phi_{\bf F} }^2 \ge \frac{4}{\pi^2 B_0} \| \phi_{\bf F} \|_{H_{h}^2}^2  \ge \frac{4}{\pi^2 B_0} \| \phi \|_{H_{h}^2}^2 , \label{prop 1-1-3}     
\end{eqnarray} 
so that \eqref{prop 1-0-1} (in Proposition~\ref{prop:1}) is proved with $C_1 = \frac{4}{\pi^2 B_0}$. 

Inequality \eqref{prop 1-0-2} could be proved in a similar way. The following fact is observed: 
\begin{eqnarray} 
  \| \phi \|_\infty \le \| \phi_{\bf F} \|_{L^\infty} \le C \| \phi_{\bf F} \|_{H_{h}^2} \le C \| \phi \|_{H_{h}^2} ,   \label{prop 1-2}     
\end{eqnarray} 
in which the first step is based on the fact that, $\phi$ is the grid interpolation of the continuous function $\phi_{\bf F}$, the second step comes from the Sobolev embedding, while the last step comes from the the estimates in Lemma~\ref{lem:A.1}.  
 
For the proof of \eqref{prop 1-0-3}, the last inequality in Proposition~\ref{prop:1}, the following lemma is needed, which gives a bound of the discrete $\ell^p$ (with $p=4, 6$) norm of the grid functions $\phi$ and $f$, in terms of the continuous $L^p$ norm of its continuous version $f_{\bf F}$. 

\begin{lem} \label{lem:A.2} 
  For $\phi \in {\mathcal C}_{\rm per}$, $f \in {\mathcal V}_{\rm per}$,  we have 
\begin{eqnarray} 
  \| \phi \|_p \le \sqrt{\frac{p}{2}} \| \phi_{\bf F} \|_{L^p} ,  \quad 
  \| f \|_p \le \sqrt{\frac{p}{2}} \| f_{\bf F} \|_{L^p} ,  \quad \mbox{with $p = 4, 6$} .  \label{lemma A.2-0} 
\end{eqnarray} 
\end{lem}

\begin{proof} 
For simplicity of presentation, we only present the analysis for $\| f \|_p \le \sqrt{\frac{p}{2}} \| f_{\bf F} \|_{L^p}$; the analysis for $\phi$ could be carried out in the same fashion. And also, we are focused on the case of $p=4$. The case with $p=6$ could be handled in a similar, yet more tedious way. 

We denote the following grid function 
\begin{equation} 
  g_{i+\hf,j+\hf} = \left( f_{i+\hf,j+\hf} \right)^2 .  \label{lemma A.2-1} 
\end{equation} 
A direct calculation shows that 
\begin{equation} 
  \nrm{ f }_4 = \left( \nrm{ g }_2 \right)^{\frac12} . \label{lemma A.2-2} 
\end{equation}  
Note that both norms are discrete in the above identity. Moreover, we assume the grid function $g$ has a discrete Fourier expansion as 
\begin{equation}
  g_{i+\hf,j+\hf} = \sum_{\ell,m=-K}^{K}
   (\hat{g}^N_c)_{\ell,m} \mathrm{e}^{2\pi {\rm i}  (\ell x_{i+1/2} + m y_{j+\hf} )} ,  
  \label{lemma A.2-3}
\end{equation}
and denote its continuous version as 
\begin{equation}
  G (x,y) = \sum_{\ell,m=-K}^{K}
   (\hat{g}^N_c)_{\ell,m} \mathrm{e}^{2\pi {\rm i} (\ell x + m y )}  
 \in {\cal P}_{K} .  
  \label{lemma A.2-4}
\end{equation}
With an application of the Parseval equality at both the discrete and continuous levels, we have 
\begin{equation} 
  \nrm{ g }_2^2 = \nrm{ G }^2 
  = \sum_{\ell,m=-K}^{K}  \left|  (\hat{g}^N_c)_{\ell,m}  \right|^2  .  \label{lemma A.2-5}
\end{equation}

  On the other hand, we also denote 
\begin{equation} 
  H (x,y) = \left( f_{\bf F} (x,y) \right)^2 
  = \sum_{\ell,m=-2K}^{2K}
   (\hat{h}^N)_{\ell,m} \mathrm{e}^{2\pi {\rm i} (\ell x + m y )}  
 \in {\cal P}_{2K} .  \label{lemma A.2-6}
\end{equation}
The reason for $H \in {\cal P}_{2K}$ is because $f_{\bf F} \in {\cal P}_{K}$. We note that $H \ne G$, since $H \in {\cal P}_{2K}$, while $G \in {\cal P}_{K}$, although $H$ and $G$ have the same interpolation values on at the numerical grid points $(x_{i+\hf}, y_{j+\hf})$. In other words, $g$ is the interpolation of $H$ onto the numerical grid point and $G$ is the continuous version of $g$ in ${\cal P}_{K}$. As a result, collocation coefficients $\hat{g}_c^N$ for $G$ are not equal to $\hat{h}^N$ for $H$, due to the aliasing error. In more detail, for $- K \le \ell, m \le K$, we have the following representations: 
\begin{eqnarray} 
  ( \hat{g}_c^N )_{\ell,m} =  \left\{  \begin{array}{l} 
      (\hat{h}^N)_{\ell,m} + (\hat{h}^N)_{\ell+N,m} 
  + (\hat{h}^N)_{\ell,m+N} + (\hat{h}^N)_{\ell+N,m+N} ,  \, \,  
   \ell < 0 , m < 0 , 
\\
      (\hat{h}^N)_{\ell,m} + (\hat{h}^N)_{\ell+N,m} ,  \, \,  
   \ell < 0 , m = 0 ,  
\\
   (\hat{h}^N)_{\ell,m} + (\hat{h}^N)_{\ell+N,m} 
  + (\hat{h}^N)_{\ell,m-N} + (\hat{h}^N)_{\ell+N,m-N} ,  \, \,  
  \ell < 0 , m > 0   , 
\\
   (\hat{h}^N)_{\ell,m} + (\hat{h}^N)_{\ell -N,m} 
  + (\hat{h}^N)_{\ell,m-N} + (\hat{h}^N)_{\ell-N,m-N} ,  \, \,  
  \ell > 0 , m > 0 , 
\\
   (\hat{h}^N)_{\ell,m} + (\hat{h}^N)_{\ell-N,m} ,  \, \,  
  \ell > 0 , m = 0 , 
\\
   (\hat{h}^N)_{\ell,m} + (\hat{h}^N)_{\ell-N,m} 
  + (\hat{h}^N)_{\ell,m+N} + (\hat{h}^N)_{\ell-N,m+N} ,  \, \,  
  \ell > 0 , m < 0 , 
\\ 
      (\hat{h}^N)_{\ell,m} + (\hat{h}^N)_{\ell,m+N} ,  \, \,  
   \ell = 0 , m < 0 , 
\\
      (\hat{h}^N)_{\ell,m}  ,  \, \,  
   \ell = 0 , m = 0 ,  
\\
   (\hat{h}^N)_{\ell,m} + (\hat{h}^N)_{\ell,m-N}  ,  \, \,  
  \ell = 0 , m > 0  . 
\end{array}  \right.  \label{lemma A.2-7}
\end{eqnarray}
With an application of Cauchy inequality, it is clear that 
\begin{equation} 
  \sum_{\ell,m=-K}^{K}  
  \left|  (\hat{g}^N_c)_{\ell,m}  \right|^2  
  \le 4 \left| \sum_{\ell,m=-2K}^{2K}
   (\hat{h}^N)_{\ell,m}  \right|^2 .  \label{lemma A.2-8}
\end{equation}

  Meanwhile, an application of Parseval's identity to the Fourier expansion (\ref{lemma A.2-6}) gives 
\begin{equation} 
  \nrm{ H }^2 = \left| \sum_{\ell,m=-2K}^{2K}
   (\hat{h}^N)_{\ell,m}  \right|^2 .  \label{lemma A.2-9}
\end{equation}
Its comparison with (\ref{lemma A.2-5}) indicates that 
\begin{equation} 
  \nrm{ g }_2^2 = \nrm{ G }^2 
   \le 4 \nrm{ H }^2  ,  \quad \mbox{i.e.} \, \, 
  \nrm{ g }_2 \le 2 \nrm{ H } ,  \label{lemma A.2-10}
\end{equation}
with the estimate (\ref{lemma A.2-8}) applied. Meanwhile, since $H (x,y) = \left( f_{\bf F} (x,y) \right)^2$, we have 
\begin{equation} 
  \nrm{ f_{\bf F} }_{L^4} = \left( \nrm{ H }_{L^2} \right)^{\frac12} . \label{lemma A.2-11}
\end{equation}
Therefore, a combination of (\ref{lemma A.2-2}), (\ref{lemma A.2-10}) and (\ref{lemma A.2-11}) results in 
\begin{equation} 
  \nrm{ f }_4 = \left( \nrm{ g }_2 \right)^{\frac12} 
  \le \left( 2 \nrm{ H }_{L^2} \right)^{\frac12} 
  \le \sqrt{2} \nrm{ f_{\bf F} }_{L^4}  . \label{lemma A.2-12}
\end{equation}
This finishes the proof of (\ref{lemma A.2-0}) with $p=4$, the inequality with $p=6$ could be proved in the same fashion. 
\end{proof} 

Now we proceed into the proof of \eqref{prop 1-0-3} in Proposition~\ref{prop:1}. 

\begin{proof} 
  We begin with an application of (\ref{lemma A.2-0}) in Lemma~\ref{lem:A.2}: 
\begin{eqnarray} 
  \| \mD_x \phi \|_6 = \| f \|_6 \le  \sqrt{3} \| f_{\bf F} \|_{L^6} .  \label{prop 1.1-1} 
\end{eqnarray}  
Meanwhile, using the fact that $\overline{ f _{\bf F} }=0$, we apply the 2-D Sobolev inequality and get 
\begin{eqnarray} 
  \| f_{\bf F} \|_{L^6} \le B_0^{(1)} \| f_{\bf F} \|_{H^1} \le C ( \| f_{\bf F} \| + \| \nabla f_{\bf F} \| )   .  \label{prop 1.1-2} 
\end{eqnarray} 
Moreover, the estimates (\ref{lemma A.1-1})-(\ref{lemma A.1-3}) (in Lemma~\ref{lem:A.1}) indicate that 
\begin{eqnarray} 
  && 
  \nrm{f_{\bf F} } \le \nrm{\partial_x \phi_{\bf F} }  \le  \frac{\pi}{2} \| \nabla_h \phi \|_2 , 
  \label{prop 1.1-3-1}
\\
  && 
  \nrm{\partial_x f_{\bf F} } \le \nrm{\partial_x^2 \phi_{\bf F} }   
  \le M_0 \nrm{\Delta \phi_{\bf F} }  \le  \frac{\pi^2 M_0}{4} \| \Delta_h \phi \|_2 , 
  \label{prop 1.1-3-2}  
\\
  && 
  \nrm{\partial_y f_{\bf F} } \le \nrm{\partial_x \partial_y \phi_{\bf F} }   
  \le M_0 \nrm{\Delta \phi_{\bf F} }  \le  \frac{\pi^2 M_0}{4} \| \Delta_h \phi \|_2 , 
  \label{prop 1.1-3-3}  
\\
  && 
  \mbox{so that} \, \, \, 
  \nrm{f_{\bf F} } + \nrm{\nabla f_{\bf F} }  \le  \frac{\sqrt{2} \pi^2 M_0}{4} ( \| \nabla_h \phi \|_2 + \| \Delta_h \phi \|_2 ) , 
  \label{prop 1.1-3-4} 
\end{eqnarray} 
in which the following elliptic regularity estimate is applied:
\begin{eqnarray} 
  \nrm{\partial_x^2 \phi_{\bf F} }  ,  \nrm{\partial_x \partial_y \phi_{\bf F} }   
  \le M_0 \nrm{\Delta \phi_{\bf F} } . \label{prop 1.1-4}   
\end{eqnarray} 
Therefore, a substitution of (\ref{prop 1.1-3-2}), (\ref{prop 1.1-3-4}) and (\ref{prop 1.1-2}) into (\ref{prop 1.1-1}) results in 
\begin{eqnarray} 
   \| \mD_x \phi \|_6  \le \frac{\sqrt{6} \pi^2 M_0 B_0^{(1)}}{4}  \| \phi \|_{H_{h}^2}  .    \label{prop 1.1-5} 
\end{eqnarray} 

The estimate for $\| D_y \phi \|_6$ could be derived in the same fashion: 
\begin{eqnarray} 
   \| \mD_y \phi \|_6  \le \frac{\sqrt{6} \pi^2 M_0 B_0^{(1)}}{4}  \| \phi \|_{H_{h}^2} .  \label{prop 1.1-6} 
\end{eqnarray}  

As a consequence, \eqref{prop 1-0-3} is valid, by setting $C = \sqrt{2} B_0^{(1)}$. The proof of Proposition~\ref{prop:1} is complete. 
\end{proof}

	\bibliographystyle{plain}
	\bibliography{mbe}
	\end{document}